\documentclass{article}

\usepackage{color}
\usepackage{graphicx}
\usepackage{url}
\usepackage{amsmath}
\usepackage{amssymb,amsthm,amsfonts}
\usepackage[ruled, linesnumbered]{algorithm2e}

\newtheorem{theorem}{Theorem}

\newtheorem{proposition}{Proposition}
\newtheorem{corollary}{Corollary}
\newtheorem{lemma}{Lemma}

\newcommand{\Z}{\mathbf{Z}}
\newcommand{\T}{\mathbf{T}}
\newcommand{\BB}{\mathbf{B}}
\newcommand{\NN}{\mathbf{N}}
\newcommand{\CC}{\mathbf{C}}

\begin{document}

\title{
Periods of iterations of functions \\
with restricted preimage sizes
\footnote{An extended abstract of this work was presented 
at the 29th International Meeting on Probabilistic, 
Combinatorial and Asymptotic Methods for the Analysis 
of Algorithms (AofA 2018), Uppsala, Sweden.}
}


\author{Rodrigo S. V. Martins \\
Universidade Tecnol\'ogica Federal do Paran\'a, Apucarana, Brazil \\
{\tt email:} rodrigomartins@utfpr.edu.br \\ \\
Daniel Panario \\
Carleton University, Ottawa, Canada \\
{\tt email:} daniel@math.carleton.ca 
\footnote{This author was partially funded by NSERC of Canada.} \\ \\
Claudio Qureshi \\
Universidade Estadual de Campinas, Brazil \\
{\tt email:} cqureshi@ime.unicamp.br 
\footnote{This author was partially funded by FAPESP grant 2015/26420-1.} \\ \\
Eric Schmutz \\
Drexel University, Philadelphia, USA \\
{\tt email:} Eric.Jonathan.Schmutz@drexel.edu}

\date{}
\maketitle

\begin{abstract}
Let $[n] = \{1, \dots, n\}$ and let $\Omega_n$ be the set of all mappings from $[n]$ to itself. Let $f$ be a random uniform element of $\Omega_n$ and let $\T(f)$ and $\BB(f)$ denote, respectively, the least common multiple and the product of the length of the cycles of $f$. Harris proved in 1973 that $\log \T$ converges in distribution to a standard normal distribution and, in 2011, Schmutz obtained an asymptotic estimate on the logarithm of the expectation of $\T$ and $\BB$ over all mappings on $n$ nodes.
We obtain analogous results for random uniform mappings on $n = kr$ nodes with preimage sizes restricted to a set of the form $\{0,k\}$, where $k = k(r) \geq 2$. This is motivated by the use of these classes of mappings as heuristic models for the statistics of polynomials of the form $x^k + a$ over the integers modulo $p$, with $p \equiv 1 \pmod k$. We exhibit and discuss our numerical results on this heuristic.
\end{abstract}

\section{Introduction} \label{sec:introduction}

Let $f: [n] \to [n]$ be a function of a finite set to itself. 
The iterations of mappings have attracted interest in recent years due to applications in areas such as physics, biology, coding theory and cryptography. 
We highlight Pollard's factorization method for integers, which is based on iterations of quadratic polynomials over finite fields. The adaptation of Pollard's method to the discrete logarithm problem also relies on iterations of mappings.

In this work we focus on asymptotic results on the cycle structure of these dynamical systems. Let $f = f^{(0)}$ be a mapping on $n$ elements and consider the sequence of functional compositions $f^{(m)} = f \circ f^{(m-1)}$, $m \geq 1$. Since there are finitely many mappings on $n$ elements, there exists an integer $T$ such that $f^{(m + T)} = f^{(m)}$ for all $m \geq n$. The least integer $T = \T(f)$ satisfying this condition equals the order of the permutation obtained by restricting the mapping $f$ to its cyclic vertices.
Erd\"os and Tur\'an proved in \cite{ErdosTuran1967lognormality} that the logarithm of the corresponding random variable defined  over the symmetric group $S_n$ converges in distribution to a standard normal distribution, when properly centered and normalized.
By adapting Erd{\H o}s and Tur{\' a}n's ``statistical group theory approach'' \cite{ErdosTuran1967lognormality},
Harris was able to prove that the normalized random variable $(\log \T - \mu_n^*)/\sigma_n^*$, where $\mu_n^* = \frac{1}{2} \log^2 \sqrt{n}$ and $\sigma_n^* = \frac{1}{ \sqrt{3} } \log^{3/2} \sqrt{n}$, defined over the space of mapping with uniform distribution, converges in distribution to a standard normal distribution \cite{harris1973DistributionMappings}.
The expected value of $\T$ was estimated in \cite{schmutz2011period}:
\begin{equation} \label{schmutz_estimate_T}
\log \mathbb{E}_n( \T ) = k_0 \sqrt[3]{ \frac{n}{ \log^2 n } } \big( 1 + o(1) \big),
\end{equation}
where $k_0$ is a constant determined explicitly that is approximately $3.36$. 
The parameter $\T$ can be proven to be the least common multiple of the cycle lengths of the components of the functional graph of $f$. If $\BB(f)$ is the product of all cycle lengths of $f$ including multiplicities, then one might consider $\BB$ as an approximation for $\T$. For instance, Proposition 1.2 of \cite{schmutz2011period} implies that, for any $\delta > 0$, the sequence of random variables defined by
$$X_n = \frac{\log \BB - \log \T}{ \log^{1+\delta} n }, \quad n \geq 1,$$
converges in probability to zero. However, it is proved in \cite{schmutz2011period} that the expectation of $\BB$ deviates significantly from the expectation of $\T$:
\begin{equation} \label{schmutz_estimate_B}
\log \mathbb{E}_n( \BB ) = \frac{3}{2} \sqrt[3]{n} \big( 1 + o(1) \big).
\end{equation}

In this paper we derive similar results for the classes of $\{0,k\}$-mappings, $k \geq 2$, defined as mappings $f: [n] \to [n]$ such that $|f^{-1}(y)| \in \{0,k\}$ for all $y \in [n]$. We derive our results in the case where $k$ is allowed to approach infinity together with $n$. 
This might be desirable, for example, when modeling polynomials whose degree depends on the size of the prime $p$; see \cite{BurnetteSchmutz2015PeriodsRationalFunctions} for an example where this occurs.
We obtain asymptotic estimates for the logarithm of the expected value of $\T$ and $\BB$ over $\{0,k\}$-mappings on $n$ nodes. We also prove an analogue of Harris' result \cite{harris1973DistributionMappings} for $\{0,k\}$-mappings, that is, we prove that $\log \T$ converges in distribution to a standard normal distribution, when properly centered and normalized. An analogous result is obtained for the parameter $\BB$.

By now there is a rather large literature on the asymptotic distribution of random variables defined on mappings with indegree restrictions.
One motivation is methodological. Random mappings are important examples that serve as benchmarks for both probabilistic and analytic methods.  On the analytic side,  combinatorial methods can be used to identify generating functions whose coefficients are the quantities of interest. In many cases
it is  possible to estimate the coefficients asymptotically using complex analysis. A standard reference is 
\cite{FlajoletSedgewick2009AnalyticCombinatorics},  which includes several  applications to random mappings; see also \cite{FlajoletOdlyzko1990RandomMappings, kolchin1986RandomMappings}.
In another direction,  random mappings correspond to  a large class of random graphs $G_{f}$ for which the joint distribution of components sizes can be realized as independent random variables, conditioned on the number of vertices that the graph has.  Stein's method and coupling have been used to prove strong and general results \cite{ArratiaBarbourTavare2000LimitsStructures, ArratiaBarbourTavare2003LogarithmicStructures}.  One application of this 
theory is a generalization of the theorem of Harris \cite{harris1973DistributionMappings} that was  mentioned above. However the proofs in our paper are elementary, and do not directly use any of these probabilistic techniques.

The research on random mappings with such restrictions is motivated also by the Brent-Pollard heuristic, where one uses these objects as a model for the statistics of polynomials. It was introduced by Pollard in the analysis of his factorization method: he conjectured that quadratic polynomials modulo large primes behave like random mappings with respect to their average rho length \cite{pollard1975factorization}. However, the indegree distribution of a class of mappings impacts the asymptotic distribution of a number of parameters \cite{ArneyBender1982mappings}. Since it is known that the functional graph of a quadratic polynomial over $\mathbb{F}_p$ has just one node with indegree $1$ and the remaining nodes are split in half between indegrees $0$ or $2$, $\{0,2\}$-mappings could provide a better heuristic model for quadratic polynomials; see \cite{MartinsPanario2015heuristic} for a discussion of alternative models for the Brent-Pollard heuristic.
Furthermore, the class of $\{0,k\}$-mappings also provides a good heuristic model for polynomials of the form $x^k + a \in \mathbb{F}_p[x]$ with $p \equiv 1 \pmod k$. This heuristic model was used in \cite{BrentPollard1981heuristic} to predict that Pollard's method is sped up in some cases if these polynomials are used, eventually leading to the factorization of the eighth Fermat number. 
We exhibit our numerical results on the behavior of $\T$ and $\BB$ over different classes of polynomials over finite fields and investigate different classes of mappings as heuristic models for the behavior of $\T$ and $\BB$ over these classes of polynomials.

This paper is organized as follows. In Section \ref{section_preliminary_results} we establish our notation and present the basic results that are needed for our main theorems; they concern mostly the distribution of the parameter that corresponds to the number of cyclic vertices of a mapping. In Section \ref{section_expected_value_of_T} we prove an asymptotic estimate for the expectation of $\T$ over $\{0,k\}$-mappings. An analogous result for the parameter $\BB$ is presented in Section \ref{section_expected_value_of_B}. 
In Section \ref{section_lognormality} we prove that the logarithm of both parameters $\T$ and $\BB$, when properly centered and normalized, converge in distribution to a standard normal distribution; we also prove in this section that $\log \BB - \log \T$, when properly normalized, converges in probability to zero. In Section \ref{section_heuristics} we present theoretical and numerical results concerning the use of classes of $\{0,k\}$-mappings as heuristic models for certain classes of polynomials.

\section{Preliminary Results} \label{section_preliminary_results}

For $f$ a mapping, let ${\cal Z} = {\cal Z}(f)$ be the set of cyclic nodes of $f$ and let $\Z = | {\cal Z} |$. In the proof of our asymptotic results we make extensive use of the law of total probability, splitting the space of random uniform $\{0,k\}$-mappings according to the value that the random variable $\Z$ assumes.
In this section we present and derive basic results concerning the distribution of this random variable over $\{0,k\}$-mappings.
To avoid confusion, we index probabilities and expected values by the set of allowed indegrees of the class of mappings in question: $\mathbb{N}$ in the unrestricted case \cite{schmutz2011period} or $\{0,k\}$ in our case. For example, the expected value of $\T$ over all mappings on $n$ nodes is denoted by $\mathbb{E}_n^\mathbb{N}(\T)$, whereas $\mathbb{E}_n^{ \{0,k\} }(\T)$ denotes the expectation of $\T$ over $\{0,k\}$-mappings on $n$ nodes.

The following theorem gives an exact result on the distribution of $\Z$ over $\{0,k\}$-mappings \cite{RubinSitgreaves1953RandomMappings}.
We note that a $\{0,k\}$-mapping $f$ of size $n$ satisfies $n = kr$, where $r \geq 1$ denotes the cardinality of the range of $f$. Also, the coalescence of $f$, defined as the variance of its distribution of indegrees under uniform distribution, satisfies
$$\lambda = \lambda(f) = \sum_{y \in [n]} \frac{| f^{(-1)}(y) |^2}{n} - 1 = r \frac{k^2}{n} - 1 = k-1.$$

\begin{theorem}[Equation (3.17) of \cite{RubinSitgreaves1953RandomMappings}] \label{theorem_distribution_Z_0,k}
Let $n = kr$, $\lambda = k-1 \geq 1$ and $1 \leq m \leq r$. A random uniform $\{0,k\}$-mapping on $n$ nodes has exactly $m$ cyclic nodes with probability
$$\mathbb{P}_n^{ \{0,k\} } ( \Z = m ) = \lambda k^{m-1} {r-1 \choose m-1} {n-1 \choose m}^{-1}.$$
\end{theorem}

It is possible to extend the quantity above to real numbers using the Gamma function $\Gamma( \cdot )$, since $n! = \Gamma(n+1)$ for any integer $n \geq 1$ (see Chapter 6 of \cite{AbramowitzStegun1965HandbookMathFunctions}):
\begin{equation} \label{distribution_Z_gamma_function}
\mathbb{P}_n^{ \{0,k\} } ( \Z = m ) = 
\lambda m k^{m-1} \frac{ \Gamma(r) }{ \Gamma(r-m+1) } \frac{ \Gamma(n-m) }{ \Gamma(n) }.
\end{equation}

In this work we consider $\{0,k\}$-mappings on $n = kr$ elements, where $r$ denotes the size of their range and $k = k(r)$ is a sequence of integers satisfying $k \geq 2$ for all $r \geq 1$. Although $n(r)$ and $k(r)$ are functions of $r$, we omit this dependence on our notation and write simply $n$ and $k$.
We emphasize that \emph{all asymptotic calculations and results in this work are taken as $r$ approaches infinity}, unless otherwise stated. 
We assume throughout the paper that, for some $0 < \alpha < 1$, $k = o(n^{1 - \alpha})$ as $r$ approaches infinity, or equivalently, $\log n=O(\log(n/\lambda))$ where $\lambda=k-1$.

Lemma \ref{lemma_gamma_function} below combines well known facts about the Gamma function $\Gamma(z)$ and the Digamma function $\Psi(z) = \frac{d}{dz} \log \Gamma(z)$; see Chapter 6 of \cite{AbramowitzStegun1965HandbookMathFunctions}.
Lemma \ref{lemma_equatios_615_616} is a simple consequence of Lemma \ref{lemma_gamma_function} and is used in the calculations of Sections \ref{section_expected_value_of_T} and \ref{section_expected_value_of_B}, so we state here for future reference.

\begin{lemma}(Chapter 6 of \cite{AbramowitzStegun1965HandbookMathFunctions}) \label{lemma_gamma_function}
The Gamma function satisfies
$$\log \Gamma(y) = y \log y - y - \frac{1}{2} \log y + \frac{1}{2} \log (2 \pi) + o(1),$$
as $y$ approaches infinity. Moreover, let $\Psi(z)$ be the derivative of $\log \Gamma(z)$. Then, as $y$ approaches infinity,
$$\Psi(y) = \log y + O \left( \frac{1}{y} \right) \quad \mbox{and} \quad \Psi^\prime (x) = \sum_{k = 0}^\infty \frac{1}{ ( x + k )^2 }.$$
\end{lemma}

\begin{lemma} \label{lemma_equatios_615_616}
Let $\Psi(z)$ be the derivative of $\log \Gamma(z)$ and let $n,k,r$ be integers such that $n = kr$. If $x = o(r)$ then, as $r$ approaches infinity,
\begin{itemize}
\item[(i)] $\displaystyle \ \ \Psi(n - x) = 
\log n - \frac{x}{n} + O \left( \frac{x^2}{n^2} \right) + O \left( \frac{1}{n} \right)$,
\item[(ii)] $\displaystyle \ \ \Psi(r - x) 
= \log r - \frac{x}{r} + O \left( \frac{x^2}{r^2} \right) + O \left( \frac{1}{r} \right)$,
\item[(iii)] $\displaystyle \ \ \ \log \Gamma(n - x) - \log \Gamma(n) = - x \log n + \frac{ x^2 }{ 2n } + O \left( \frac{ x^3 }{ n^2 } \right) + o(1)$,
\item[(iv)] $\displaystyle \ \ \ \log \Gamma(r) - \log \Gamma(r - x + 1) = x \log r - \log r - \frac{ x^2 }{ 2r } + O \left( \frac{ x^3 }{ r^2 } \right) + o(1)$.
\end{itemize}
\end{lemma}
\begin{proof}
It follows directly from Lemma \ref{lemma_gamma_function} that
$$\Psi(n-x) = \log (n-x) + O \left( \frac{1}{n-x} \right) = \log n + \log \left( 1 - \frac{x}{n} \right) + O \left( \frac{1}{n} \right).$$
The estimate for $\Psi(n-x)$ follows from the estimate $\log (1-z) = -z + O(z^2)$, as $z$ approaches zero. The same argument proves the estimate for $\Psi(h-x)$.

We prove now items (iii) and (iv).
It follows from Lemma \ref{lemma_gamma_function} that
$$\log \Gamma(n - x) - \log \Gamma(n) =
(n - x) \log (n - x) + x - \frac{\log (n - x)}{2} - n \log n 
+ \frac{\log n}{2} + o(1).$$
We use the fact that $\log (n-x) = \log n + \log(1 - x/n)$ to obtain
$$\log \Gamma(n - x) - \log \Gamma(n) =
-x \log n + (n - x) \log \left( 1 - \frac{x}{n} \right) + x 
- \frac{1}{2} \log \left( 1 - \frac{x}{n} \right) + o(1).$$
Since $x = o(r)$ implies $x = o(n)$, we have $\log(1 - x/n) = o(1)$. The expansion of $\log(1 - z)$ then yields
$$\log \Gamma(n - x) - \log \Gamma(n) =
-x \log n + (n - x) \left( - \frac{x}{n} - \frac{ x^2 }{ 2n^2 } + O \left( \frac{ x^3 }{ n^3 } \right) \right) + x + o(1),$$
and hence,
$$\log \Gamma(n - x) - \log \Gamma(n) = - x \log n + \frac{ x^2 }{ 2n } + O \left( \frac{ x^3 }{ n^2 } \right) + o(1).$$
The estimate for $\log \Gamma(r) - \log \Gamma(r - x + 1)$ follows by the same arguments.
\hfill
\end{proof}

In the following lemma we obtain an asymptotic estimate on the distribution of $\Z$ over $\{0,k\}$-mappings; see \cite{ArneyBender1982mappings} for a similar result in a more general setting.

\begin{lemma} \label{lemma_asymptotic_distribution_Z_0,k}
Let $\lambda  = k-1$. If $m = m(r)$ is a sequence of positive integers such that $m = o(r)$, then the distribution of the number of cyclic nodes on a $\{0,k\}$-mapping on $n = kr$ nodes satisfies
$$\mathbb{P}_n^{ \{0,k\} } ( \Z = m ) = 
\frac{ \lambda m }{n} \exp \left( - \frac{ \lambda m^2 }{ 2n }  + O \left( \frac{ m^3 }{ r^2 } \right) + o(1) \right),$$
as $r$ approaches infinity. Moreover, if $m = m(r)$ is a sequence of real numbers such that $m \to \infty$ and $m = o(r)$ as $r \to \infty$, then $\mathbb{P}_n^{ \{0,k\} } ( \Z = \lfloor m \rfloor )$, $r \geq 1$, is asymptotically equivalent to the quantity above as $r$ approaches infinity. 
\end{lemma} 
\begin{proof}
Let $S_1 = \log \Gamma(r) - \log \Gamma(r - m + 1)$ and $S_2 = \log \Gamma(n - m) - \log \Gamma(n)$.
It follows from Equation (\ref{distribution_Z_gamma_function}) that
$$\log \mathbb{P}_n^{ \{0,k\} } ( \Z = m ) = \log \left( \frac{\lambda m}{k} \right) + m \log k + S_1 + S_2.$$
Since $m \log k + m \log r - m \log n = 0$, Lemma \ref{lemma_equatios_615_616} implies
$$\log \mathbb{P}_n^{ \{0,k\} } ( \Z = m ) =
\log \left( \frac{\lambda m}{k} \right) - \log r - \frac{ m^2 }{ 2r } 
+ \frac{ m^2 }{ 2n } + O \left( \frac{ m^3 }{ r^2 } \right) + o(1).$$
The first result follows from $n = kr$ and $\lambda = k-1$. 

Let $m = m(r)$ be a sequence of real numbers such that $m \to \infty$ and $m = o(r)$ as $r \to \infty$. We note that $\lfloor m \rfloor = m + O(1) = m (1 + O(m^{-1})) = m (1 + o(1))$. Since $(1 + O(m^{-1}))^2 = 1 + O(m^{-1})$, using the first part of the lemma we obtain
\begin{align*}
  & \mathbb{P}_n^{ \{0,k\} } ( \Z = \lfloor m \rfloor ) \\
= & 
\frac{ \lambda m }{n} (1 + o(1)) \exp \left( - \frac{ \lambda m^2 }{2n} (1 + O(m^{-1}))  + O \left( \frac{ m^3 }{ r^2 } \right) + o(1) \right) \\
= &
\frac{ \lambda m }{n} \exp \left( - \frac{ \lambda m^2 }{2n} + O \left( \frac{\lambda m}{n} \right)  + O \left( \frac{ m^3 }{ r^2 } \right) + o(1) \right).
\end{align*}
The second result follows from the fact that $\lambda m/n = O(m/r) = o(1)$.
\hfill
\end{proof}

\begin{lemma} \label{lemma_unimodal}
Let $n,k \geq 2$ be fixed integers such that $n = kr$ for some $r \geq 1$. Then there exists a positive  real number $m_{\#}$ such that the sequence $(z_m)_m$ defined by
$$z_m = \mathbb{P}_n^{\{ 0,k\}}(\Z = m), \quad m \geq 1,$$
is increasing for $m < m_{\#}$ and decreasing for $m > m_{\#}.$
Furthermore, $m_\#$ verifies $\lambda m_{\#} (m_{\#} +1) = n$ and $m_{\#} = \sqrt{n/\lambda} + O(1)$.
\end{lemma}
\begin{proof} 
First we note that $z_m = 0$ for $m > r$.
Let $R_m= z_{m+1} / z_m$, $1 \leq m \leq r$, be the ratio of consecutive probabilities. It suffices to find a number $m_{\#}$ that $R_m \geq 1$ for $1 \leq m < m_\#$ and that $R_m\leq 1$ for $m_\# \leq m \leq r$.  
Using Theorem \ref{theorem_distribution_Z_0,k} we obtain
$$R_m= \frac{n - km}{n - m - 1}\frac{m+1}{m}.$$
We note that $R_m < 1$ is equivalent to $(n - km)(m+1) < m(n - m - 1)$. Since $n = kr$ and $\lambda = k-1$, this is equivalent to
\begin{equation} \label{ratios}
\frac{n}{\lambda}< m(m+1).
\end{equation}
The function $m \mapsto m(m+1)$ assumes the value $0$ if $m = 0$ and it approaches infinity when so does $m$. This function is monotone increasing, hence there exists a positive real number $m_{\#}$ such that $R_m \geq 1$ for $m < m_{\#}$ and $R_m < 1$ for  $m > m_{\#}$. Since $r(r+1) \geq n/\lambda$, then $m_\# \leq r$. This proves the first part of the  lemma.  We can explicitly calculate $m_\#$ by solving Equation (\ref{ratios}) as a quadratic equation for $m$:
\begin{equation}
\label{msharp}
m_\# = - \frac{1}{2} + \frac{1}{2} \left( 1 + 4 \cdot \frac{n}{\lambda} \right)^{1/2} =
\sqrt{ \frac{n}{\lambda} } \left( 1 + O \left( \sqrt{\frac{\lambda}{n} }\right) \right).
\end{equation}
\hfill
\end{proof}

In Section \ref{section_lognormality} we split the range $[1,n]$ of possible values for $\Z$ in three intervals using sequences $\xi_1 = \xi_1(n)$ and $\xi_2 = \xi_2(n)$, where $[\xi_1, \xi_2]$ defines a sequence of intervals that becomes increasingly narrow around the mode $m_\#$ (see Lemma \ref{lemma_unimodal}). 
We prove in Lemma \ref{lemma_concentration} below that $\Z$ is concentrated in the interval $[\xi_1, \xi_2]$.
We observe that, for $k \geq 2$ fixed, it is proved in \cite{ArneyBender1982mappings} that $\mathbb{E}_n^{ \{0,k\} }(\Z) \sim \sqrt{\pi n/2 \lambda}$, hence the mode $m_\#$ has the same order of growth than the expectation of $\Z$. 

\begin{lemma} \label{lemma_concentration} 
Let $\varepsilon_n = \log^{-3/4}( \sqrt{n/\lambda} )$, $\xi_1 = m_{\#}^{1-\varepsilon_n}$,  and $\xi_2=m_{\#}^{1+\varepsilon_n}$. If $c$ is any positive constant  less than $2^{3/4}$, then  for all sufficiently large $n$, 
\begin{itemize}
\item[(i)] $\ \ \mathbb{P}_n^{\{ 0,k\}} (\Z < \xi_1) )\leq  \exp \left( - c \log^{1/4} \left( \frac{n}{\lambda }\right) \right),    $
\item[(ii)] $\ \ \mathbb{P}_n^{\{ 0,k\}} (\Z > \xi_2) \leq \exp \left( -  c \log^{1/4}  \left( \frac{n}{\lambda }\right) \right)$, and
\item[(iii)] $\ \ \ \mathbb{P}_n^{\{ 0,k\}} (\xi_1 \leq \Z \leq \xi_2) \geq 1- 2\exp \left( - c \log^{1/4} \left( \frac{n}{\lambda }\right) \right).$
\end{itemize}
\end{lemma}
\begin{proof}
Since $\xi_1< m_{\#}$,  Lemma \ref{lemma_unimodal} implies that the probabilities are increasing on the first interval:
$$\mathbb{P}_n^{\{0,k\}} (\Z  < \xi_1) = \sum_{m < \xi_1} \mathbb{P}_n^{\{ 0,k\}} (\Z = m) \leq 
\sum_{m \leq \xi_1} \mathbb{P}_n^{\{ 0,k\}} (\Z = \lfloor \xi_1\rfloor),$$
and hence $\mathbb{P}_n^{\{ 0,k\}} (\Z < \xi_1) \leq \xi_1 \mathbb{P}_n^{\{ 0,k\}} (\Z = \lfloor \xi_1 \rfloor)$.
Therefore, by Lemma \ref{lemma_asymptotic_distribution_Z_0,k},
\begin{equation} \label{cdfxi1}
\mathbb{P}_n^{\{ 0,k\}} (\Z < \xi_1) \leq 
\frac{ \lambda \xi_1 ^2 }{n} \exp \left( -\frac{ \lambda\lfloor \xi_1\rfloor^{2}}{2n}
           + O \left( \frac{\xi^{3}_1}{r^2}\right)+o(1) \right).
\end{equation}
To estimate the right hand side of (\ref{cdfxi1}), first observe that, 
from the definition of $m_\#$, we have
$ \frac{n}{\lambda}= m_\# ( m_\# +1)\geq  m^{2}_\# $. Therefore
 $m_\# \leq \sqrt{\frac{n}{\lambda}}$ for all $n$, and 

\begin{equation}
\label{nonasymptotic}
\frac{ \lambda \xi_1^2 }{n} \leq \frac{\lambda}{n} \left( \sqrt{\frac{n}{\lambda}} \right)^{2(1 - \varepsilon_n)} =
\left( \frac{n}{\lambda} \right)^{-\varepsilon_n} =
\exp \left( - 2^{3/4} \log^{1/4} \left( \frac{n}{\lambda } \right)\right).
\end{equation}
In the exponent on the right hand side of $(\ref{cdfxi1})$, we have that
 $-\frac{ \lambda\lfloor \xi_1\rfloor^{2} }{2n}\leq 0$. Since $r^{-1} = O( \lambda/n )$, it is also straightforward to check
 that $\frac{\xi^{3}_1}{r^2}\rightarrow 0$.
Thus 
\[ \mathbb{P}_n^{\{ 0,k\}} (\Z < \xi_1) \leq \exp \left( - 2^{3/4} \log^{1/4} \left( \frac{n}{\lambda } \right)\right) \exp(o(1)).\]
It follows that if $c$ is any positive constant less than 
$2^{3/4} $, then for all sufficiently large $n$, we have the inequality in part (i) of the lemma:
\begin{equation}
\label{lowertail}
\mathbb{P}_n^{\{ 0,k\}} (\Z < \xi_1) \leq \exp \left( - c  \log^{1/4} \left( \frac{n}{\lambda }\right) \right).
\end{equation}

In order to estimate the upper tail, we  again use  monotonicity. Because of the restriction on
in-degrees, we know that the number of cyclic vertices is at most $\frac{n}{k}$.
Using this and  Lemma \ref{lemma_unimodal} we get
$\mathbb{P}_n^{\{ 0,k\}} (\Z > \xi_2) \leq \frac{n}{k} \mathbb{P}_n^{\{ 0,k\}} (\Z = \lfloor \xi_2 \rfloor)$.
Therefore, using Lemma \ref{lemma_asymptotic_distribution_Z_0,k} and applying logarithm on both sides we obtain
\begin{equation} \label{cdfxi2}
\log \mathbb{P}_n^{\{ 0,k\}}(   \Z  > \xi_2 )\leq  \log\left( \frac{ \lambda \xi_{2}}{k}\right) -\frac{\lambda \lfloor \xi_2\rfloor^{2}}{2n} + O \left( \frac{\xi^{3}_2 }{r^2} \right)  + o(1).
\end{equation}
Using the formula $m_\# = \sqrt{ \frac{n}{\lambda} } \left( 1 + O \left( \sqrt{\frac{\lambda}{n}} \right) \right)$ from (\ref{msharp}), 
 together with the definition $\xi_2:=m_{\#}^{1+\varepsilon_n}$, we get
\begin{equation} \label{xi_2^2/n}
-\frac{ \lambda \lfloor\xi_{2} \rfloor^{2} }{n}  \sim
-\exp \left( 2^{3/4} \log^{1/4} \left( \frac{n}{\lambda }\right) \right).
\end{equation}
This large negative term accounts for the fact that the upper tail bound (\ref{cdfxi2}) is small.
In (\ref{cdfxi2}), the first term is negligible since
\[\log\left( \frac{ \lambda \xi_2 }{k}\right) \leq \log \xi_2 
=  O\left(\log \frac{n}{\lambda}\right).\]
In (\ref{cdfxi2}), the term $O \left( \frac{\xi^{3}_2 }{r^2} \right)$ 
is also negligible. To see this, recall that $m_\# \leq \sqrt{\frac{n}
{\lambda}}$. Since $r= \frac{n}{\lambda+1}\geq \frac{n}{2\lambda}$, 
it follows from the definition of $\xi_{2}$ that
$$ \frac{\xi^{3}_2 }{r^2} 
 < \frac{ (\frac{n}{\lambda})^{\frac{3}{2}(1+\varepsilon_{n})}}
 { \frac{1}{4}(\frac{n}{\lambda})^{2}}
 = \left(\frac{n}{\lambda}\right)^{-\frac{1}{2}+o(1)}=o(1).$$
Thus
 \begin{equation}\label{bound2}
 \log \mathbb{P}_n^{\{ 0,k\}} (\Z > \xi_2) \leq  - \exp\left( 2^{3/4}\log^{1/4} \left( \frac{n}{\lambda }\right)\right) \cdot (1+o(1)).
\end{equation}
Part (ii) of the lemma  is a weak  consequence of (\ref{bound2}) that is 
convenient for future reference:  if $c$ is any  positive constant, then for all sufficiently large $n$,
\begin{equation}
\label{crude}
 \mathbb{P}_n^{\{ 0,k\}} (\Z > \xi_2) \leq  \exp \left( -c\log^{1/4} \left( \frac{n}{\lambda }\right) \right).
 \end{equation}
Part (iii) of the lemma follows immediately from (\ref{lowertail}), (\ref{crude}), and the fact that the sum of the three probabilities is 1.

%
%
%
%
\end{proof}

It is a well known fact that the restriction of a random unrestricted mapping $f$ to its set $\cal Z$ of cyclic nodes is a uniform random permutation of $\cal Z$, but this also holds for $ \{0,k\} $-mappings; see Lemma 1 of \cite{ArneyBender1982mappings}. We state this result below for future reference. We denote the symmetric group on $n$ elements by $S_n$ and remark that if $f: [n] \to [n]$ is a mapping such that $\Z(f) = m$, then there exists a unique permutation $\sigma_f \in S_m$ and an increasing function $\varphi_f: {\cal Z}(f) \rightarrow [m]$ such that $\varphi_f \circ f \circ \varphi_f^{-1} = \sigma_f$. We write $f\big|_{\cal Z} \equiv \sigma$ in this case.

\begin{lemma}  \label{lemma_restriction_to_Z}
Let $n = kr$, $n,k \geq 2$. Let $\cal A$ be a subset of $[n]$ with $m$ elements and let $\sigma \in S_m$. If $m \leq n/k$, then 
$$\mathbb{P}_n^{\{ 0,k\}} \left( f \big|_{\cal A} \equiv \sigma \, \big| \, {\cal Z} = {\cal A} \right) = \frac{1}{m!}.$$
\end{lemma} 

The following lemma is used in Sections \ref{section_expected_value_of_T} and \ref{section_expected_value_of_B} in our asymptotic estimates: we obtain upper and lower bounds for the expectation of $\T$ and $\BB$ in the form of item (ii) of the lemma below.

\begin{lemma} \label{lemma_calc}
Let $\langle L_{r} \rangle_{r=1}^{\infty}$ and $\langle A_{r} \rangle_{r=1}^{\infty}$ be sequences of positive real numbers. Then
the following are equivalent:
\begin{itemize}
\item[(i)] $L_{r}=A_{r}(1+o(1))$ as $r\rightarrow\infty$;
\item[(ii)] for any $\varepsilon>0$, there exists $R = R(\varepsilon)$ such that the inequalities 
$(1 - \varepsilon) A_{r} < L_{r} < (1 + \varepsilon) A_{r}$ hold for all $r > R$.
\end{itemize}
\end{lemma}
\begin{proof}
First we note that $L_{r} = A_{r}(1+o(1))$ if and only if $\frac{L_{r}}{A_{r}} - 1 = o(1)$, that is, if and only if 
$\lim\limits_{r\rightarrow\infty} (\frac{L_{r}}{A_{r}}-1) = 0$. By definition of a limit, this holds if and only if for any $\varepsilon > 0$ there exists $R = R(\varepsilon)$ such that $|\frac{L_{r}}{A_{r}}-1| < \varepsilon$ for all $r > R$. It can be easily checked that this condition is equivalent to $(1-\varepsilon_1)A_{r}<L_{r}<(1+\varepsilon_1)A_{r}$.
\hfill
\end{proof}

\section{Expected Value of $\T$} \label{section_expected_value_of_T}

In this section we obtain asymptotic estimates for $\mathbb{E}_n^{ \{0,k\} } (\T)$ following the same strategy as in \cite{schmutz2011period}, that we describe next.
We can write the expected value of $\T$ over all $\{0,k\}$-mappings as
\begin{equation}
\mathbb{E}_n^{ \{0,k\} } ( \T ) = \sum_{m = 1}^n \mathbb{P}_n^{ \{0,k\} } ( \Z = m ) \mathbb{E}_n^{ \{0,k\} } ( \T | \Z = m ).
\label{expected_value_of_T_preliminary}
\end{equation}
If we let $M_m$ be the expected order of a uniform random permutation of $S_m$, then Equation (\ref{expected_value_of_T_preliminary}) and Lemma \ref{lemma_restriction_to_Z} imply
\begin{equation}
\mathbb{E}_n^{ \{0,k\} } ( \T ) = \sum_{m = 1}^n \mathbb{P}_n^{ \{0,k\} } ( \Z = m ) M_m.
\label{expected_value_of_T}
\end{equation}
The author in \cite{schmutz2011period} combines an exact result for $\mathbb{P}_n^\mathbb{N} ( \Z = m )$ with the following lemma to estimate the expected value of $\T$ asymptotically in the case of unrestricted mappings. We use Theorem \ref{theorem_distribution_Z_0,k} for the distribution of $\Z$ over $\{0,k\}$-mappings.

\begin{lemma}(\cite{Stong1998}) \label{lemma_M_m}
Let $M_m$ be the expected order of a random permutation of $S_m$. Let $\beta_0 = \sqrt{8 I}$, where
\begin{equation} \label{I_definition_constant}
I = \int_0^\infty \log \log \left( \frac{e}{ 1 - e^{-t} } \right) dt.
\end{equation}
Then, as $m$ approaches infinity,
$$\log M_m = \beta_0 \sqrt{ \frac{m}{\log m} } + O \left( \frac{ \sqrt{m} \log \log m }{\log m} \right).$$
In particular, if $\varepsilon_1 \in (-1,0)$, $\varepsilon_2 \in (0,1)$ and $\beta_\varepsilon = \beta_0 + \varepsilon$, there exists $N_\varepsilon$ such that, for all $m > N_\varepsilon$,
$$\beta_{\varepsilon_1} \sqrt{ \frac{m}{\log m} } < \log M_m < \beta_{\varepsilon_2} \sqrt{ \frac{m}{\log m} }.$$
\end{lemma}

It is clear from Equation (\ref{expected_value_of_T}) that, if $m_*$ is the integer that maximizes $\mathbb{P}_n^{ \{0,k\} } ( \Z = m ) M_m$ for $1 \leq m \leq n$ and $m_0$ is an integer in $(1,n)$, then
$$\mathbb{P}_n^{\{ 0,k \}} ( \Z = m_0 ) M_{m_0} \leq \mathbb{E}_n^{\{ 0,k \}} ( \T ) \leq n \mathbb{P}_n^{\{ 0,k \}} ( \Z = m_* ) M_{m_*}.$$
Let $n \geq 1$ and $\varepsilon \in (-1,1)$. 
We extend the factorials in the expression for $\mathbb{P}_n^{\{ 0,k \}} ( \Z = m_* )$ in Theorem \ref{theorem_distribution_Z_0,k} using the Gamma function, as in Equation (\ref{distribution_Z_gamma_function}). Also, we bound the quantity $M_m$ for large $m$ as described in Lemma \ref{lemma_M_m}. For $\beta_\varepsilon = \beta_0 + \varepsilon$, let 
\begin{equation}
\phi_{n, \varepsilon} (x) = \lambda x k^{x - 1} \frac{ \Gamma(r) }{ \Gamma(r - x + 1) } \frac{ \Gamma(n - x) }{ \Gamma(n) } 
\exp \left( \beta_\varepsilon \sqrt{ \frac{x}{\log x} } \right).
\label{phi_n}
\end{equation}
The calculation of the maximum value that the real function $\phi_{n, \varepsilon}(x)$ assumes for $x \in (1,n)$ is a main ingredient in the proof of the asymptotic estimate on $\mathbb{E}_n^{\{ 0,k \}} ( \T )$.
In order to simplify the calculations that follow, we consider the function $\Phi_{n, \varepsilon}(x) = \log \phi_{n, \varepsilon}(x)$ and note that $x_*$ is a local maximum of $\phi_{n, \varepsilon}(x)$ if and only if it is a local maximum of $\Phi_{n, \varepsilon}(x)$. 

\begin{proposition} \label{propostion_maximizing_phi}
Let $n = kr$, $\lambda = k-1 \geq 1$ and $\varepsilon \in (-1, 1)$.
If, for some $0 < \alpha < 1$, $k = o(n^{1 - \alpha})$ as $r$ approaches infinity, then there exists a constant $c > 0$ such that, for sufficiently large $n$, the function $x \longmapsto \phi_{n, \varepsilon} (x)$ assumes a unique maximum $x_*$ for $x \in (c, r)$. Moreover, if $k_\varepsilon = \sqrt[3]{ 3^5 \beta_\varepsilon^4 }/8$, then, as $r$ approaches infinity,
$$\log \phi_{n, \varepsilon} (x_*) = 
k_\varepsilon \frac{ (n/\lambda)^{1/3} }{ \log^{2/3} (n/\lambda) } (1 + o(1)).$$
\end{proposition}
\begin{proof}
Let $\Phi_{n, \varepsilon} = \log \phi_{n, \varepsilon}$. We note that
$$\Phi_{n, \varepsilon}^\prime (x) = 
\frac{1}{x} + \log k + \frac{d}{dx} \log \left( \frac{ \Gamma(n - x) }{ \Gamma(r - x + 1) } \right)
+ \frac{ \beta_\varepsilon }{2} \sqrt{ \frac{\log x}{x} } \frac{\log x - 1}{\log^2 x},$$
and hence,
\begin{equation} \label{derivative_Phi}
\Phi_{n, \varepsilon}^\prime (x) = \log k + \frac{1}{x} + \Psi(r - x + 1) - \Psi(n - x) + 
\frac{ \beta_\varepsilon }{2 \sqrt{ x \log x } } \left( 1 - \frac{1}{\log x} \right),
\end{equation}
where $\Psi(z)$ denotes the derivative of $\log \Gamma(z)$.
In order to prove the uniqueness of the maximum of $\Phi_{n, \varepsilon}$, we note that
\begin{equation} \label{second_derivative_Phi}
\Phi_{n, \varepsilon}^{ \prime \prime }(x) = 
- \frac{1}{x^2} - \Psi^\prime (r - x + 1) + \Psi^\prime (n - x) 
- \frac{ \beta_\varepsilon }{4} (x^3 \log x)^{-1/2} 
\left( 1 - \frac{3}{ \log^2 x } \right).
\end{equation}
It follows from Lemma \ref{lemma_gamma_function} that $\Psi^\prime(y)$ is monotone decreasing for $y$ a positive real number, so $n \geq r + 1$ implies $- \Psi^\prime (r - x + 1) + \Psi^\prime (n - x) \leq 0$. We conclude using Equation (\ref{second_derivative_Phi}) that $\Phi_{n, \varepsilon}^{\prime \prime}(x) < 0$ if $1 - 3 \log^{-2} x > 0$; this condition holds if $x > c$, where $c = \exp(\sqrt{3})$.

We note that Equation (\ref{derivative_Phi}) implies that $\Phi_{n, \varepsilon}^\prime(x) = 0$ if and only if
$$\log k + \frac{1}{x} + \Psi(r - x + 1) - \Psi(n - x) 
+ \frac{ \beta_\varepsilon }{2 \sqrt{ x \log x } } \left( 1 - \frac{1}{\log x} \right) = 0.$$
We proceed heuristically in order to obtain an intuition for the asymptotic behaviour of the point $x_* \in (1,n)$ that maximizes $\Phi_{n, \varepsilon}(x)$. 
By Lemma \ref{lemma_equatios_615_616}, for $x = o(r)$ we have
\begin{equation} \label{Psi_difference_for_heuristic}
\Psi(r - x + 1) - \Psi(n-x) \sim - \log k - \frac{(k-1) x}{n}.
\end{equation}
Assume that the estimate (\ref{Psi_difference_for_heuristic}) holds as an equality; since $\lambda = k-1$, the equation $\Phi_{n, \varepsilon}^\prime(x) = 0$ is equivalent to
$$\frac{1}{x} - \frac{\lambda x}{n} + \frac{ \beta_\varepsilon }{2 \sqrt{ x \log x } } \left( 1 - \frac{1}{\log x} \right) = 0,$$
and multiplying this equation by $x$ we obtain
$$\frac{ \beta_\varepsilon }{2} \left( \frac{x}{ \log x } \right)^{1/2} \left( 1 - \frac{1}{\log x} + 
\frac{2}{ \beta_\varepsilon } \left( \frac{ \log x }{x} \right)^{1/2} \right) = \frac{\lambda x^2}{n}.$$
This is equivalent to
$$( x^3 \log x )^{1/2} = \frac{ \beta_\varepsilon }{2} \frac{n}{\lambda} 
\left( 1 - \frac{1}{\log x} + \frac{2}{ \beta_\varepsilon } \left( \frac{ \log x }{x} \right)^{1/2} \right).$$
If the function $\Phi_{n, \varepsilon}(x)$ assumes indeed a unique maximum $x_*$ in $(c,r)$, $c = \exp(\sqrt{3})$, and $x_*$ approaches infinity when so does $n$, we expect to have
$$( x^3_* \log x_* )^{1/2} = \frac{ \beta_\varepsilon }{2} \frac{n}{ \lambda } (1 + o(1)),$$
that is,
\begin{equation} \label{maximum_psi_heuristic}
x^3_* \log x_* = \frac{ \beta_\varepsilon^2 }{4} \left( \frac{n}{ \lambda } \right)^2 (1 + o(1)).
\end{equation}
We use bootstrapping to obtain an approximation for the solution of
Equation (\ref{maximum_psi_heuristic}); see Section 4.1.2 of \cite{GreeneKnuth2007AofA}.
If not for the term $\log x_*$ in Equation (\ref{maximum_psi_heuristic}), the solution would present asymptotic behavior $x_* \sim c_1 (n/\lambda)^{2/3}$ for some real number $c_1 > 0$, and thus $\log x_* \sim \frac{2}{3} \log (n/\lambda)$ as $r$ approaches infinity. Hence,
$$x^3_* \frac{2}{3} \log (n/\lambda) = \frac{ \beta_\varepsilon^2 }{4} \left( \frac{n}{ \lambda } \right)^2 (1 + o(1)),$$
that is,
$$x^3_* = \frac{ 3 \beta_\varepsilon^2 }{8} 
\frac{ (n/\lambda)^2 }{ \log (n/\lambda) } (1 + o(1)).$$
Therefore,
\begin{equation} \label{x_tilde}
x_* = \sqrt[3]{ \frac{3 \beta_\varepsilon^2 }{8} } \frac{ (n/\lambda)^{2/3} }{ \log^{1/3} (n/\lambda) }
(1 + o(1)).
\end{equation}
We prove now what was obtained heuristically in Equation (\ref{x_tilde}). We define 
\begin{equation} \label{t_tilde}
t_* = \sqrt[3]{ \frac{3 \beta_\varepsilon^2}{8} } \frac{ (n/\lambda)^{2/3} }{ \log^{1/3} (n/\lambda) }
\end{equation}
and consider, for some $\delta_n = o(1)$ to be determined, the interval $[a,b]$ where $a = a(n,\varepsilon)$, $b = b(n,\varepsilon)$ are defined by 
$$a = t_* (1 - \delta_n) \quad \mbox{and} \quad b = t_* (1 + \delta_n).$$ 
We prove using Equation (\ref{derivative_Phi}) that
\begin{equation} \label{proving_heuristic}
\Phi_{n, \varepsilon}^\prime(a) > 0 \quad \mbox{and} \quad \Phi_{n, \varepsilon}^\prime(b) < 0.
\end{equation}
Equation (\ref{proving_heuristic}) implies $x_* = t_* \big( 1 + O( \delta_n ) \big)$, as desired. We prove Equation (\ref{proving_heuristic}) using Equation (\ref{derivative_Phi}), where the last term in the expression for $\Phi_{n, \varepsilon}^\prime (b)$ is given by
$$\frac{ \beta_\varepsilon }{2} \left( \frac{1}{ b \log b } \right)^{1/2} =
\frac{ \beta_\varepsilon }{2} \left( \frac{1}{ {t_*} (1 + \delta_n) } \right)^{1/2}
\left( \frac{2}{3} \log \left( \frac{n}{\lambda} \right) + O \big( \log \log n \big) \right)^{-1/2},$$
that is,
$$\frac{ \beta_\varepsilon }{2}  \left( \frac{1}{ b \log b } \right)^{1/2} =
\frac{ \beta_\varepsilon }{2} t^{-1/2}_* \left( \frac{ 3/2 }{ (1 + \delta_n) \log (n/\lambda) } \right)^{1/2} 
\left( 1 + O \left( \frac{ \log \log n }{ \log n } \right) \right).$$
Hence, using Equation (\ref{t_tilde}),
\begin{equation} \label{confirming_heuristic_1}
\frac{ \beta_\varepsilon }{2}  \left( \frac{1}{ b \log b } \right)^{1/2} =
\sqrt[3]{ \frac{3 \lambda \beta_\varepsilon ^2 }{8 n \log (n/\lambda)} } (1 + \delta_n)^{-1/2} 
\left( 1 + O \left( \frac{ \log \log n }{ \log n } \right) \right). 
\end{equation}
Since $b = o(r)$, it follows from Lemma \ref{lemma_equatios_615_616} that
\begin{equation} \label{confirming_heuristic_2_lead_up_to}
\log k + \Psi (r - b + 1) - \Psi (n - b) = - \frac{ \lambda {t_*} }{n} (1 + \delta_n) + O \left( \frac{ t^2_* }{r^2} \right).
\end{equation}
Equations (\ref{derivative_Phi}), (\ref{confirming_heuristic_1}) and (\ref{confirming_heuristic_2_lead_up_to}) together with $\frac{1}{b}=O\left(\frac{\log\log n}{\log n}\right)$ imply
\begin{align*}
\Phi_{n, \varepsilon}^\prime (b) 
& = 
\frac{1}{b} + O\left(\frac{t_{*}^{2}}{r^2}\right)
- \frac{ \lambda {t_*}}{n} (1 + \delta_n) \\
& + 
\sqrt[3]{ \frac{3 \lambda \beta_\varepsilon ^2 }{8 n \log (n/\lambda)} } (1 + \delta_n)^{-1/2} 
\left( 1 + O \left( \frac{ \log \log n }{ \log n } \right) \right).
\end{align*}
Using $(1 + \delta_n)^{-1/2} = 1 - \frac{1}{2} \delta_n + O( \delta_n^2 )$, we obtain
\begin{align*}
\Phi_{n, \varepsilon}^\prime (b) 
& = \frac{1}{b} + O\left(\frac{t_{*}^{2}}{r^2}\right)
- \frac{ \lambda {t_*}}{n} (1 + \delta_n) \\
& + 
\sqrt[3]{ \frac{3 \lambda \beta_\varepsilon ^2 }{8 n \log (n/\lambda)} }
\left( 1 - \frac{\delta_n}{2} + O(\delta_n^2)+ O \left( \frac{ \log \log n }{ \log n } \right) \right).
\end{align*}

We note that 
\begin{align*}
\frac{1}{b} + O\left(\frac{t_{*}^{2}}{r^2}\right)
&=\frac{1}{b} (1+o(1)) = \frac{1}{t_{*}} (1+o(1))=\frac{\lambda t_{*}}{n} \cdot \frac{n}{\lambda t_{*}^2} (1+o(1)) \\
& 
=\frac{\lambda t_{*}}{n} \cdot o(n^{-\alpha/3}\log^{2/3} n).
\end{align*}

Since $\lambda t_*/n = \sqrt[3]{ 3 \lambda \beta_\varepsilon ^2 / 8n \log (n/\lambda) }$ we conclude that 
\begin{align*}
  & \Phi_{n, \varepsilon}^\prime (b) \\
= & \sqrt[3]{ \frac{3 \lambda \beta_\varepsilon ^2 }{8 n \log (n/\lambda)} }
\left( - \frac{3}{2} \delta_n + O( \delta_n^2 )+o(n^{-\alpha/3}\log^{2/3} n) + O \left( \frac{ \log \log n }{ \log n } \right) \right)\\
= & \sqrt[3]{ \frac{3 \lambda \beta_\varepsilon ^2 }{8 n \log (n/\lambda)} }
\left( - \frac{3}{2} \delta_n + O( \delta_n^2 ) + O \left( \frac{ \log \log n }{ \log n } \right) \right).
\end{align*}

We recall that $\delta_n$ is a quantity to be determined satisfying $\delta_n = o(1)$. It is of our interest to write 
$$\Phi_{n, \varepsilon}^\prime (b) = 
\sqrt[3]{ \frac{3 \lambda \beta_\varepsilon ^2 }{8 n \log (n/\lambda)} } \left( - \frac{3}{2} \delta_n + o( \delta_n ) \right),$$
as this would allow us to determine if $\Phi_{n, \varepsilon}^\prime (b)$ is positive or negative, depending on the value of $\delta_n$. To this end we choose $\delta_n = (\log \log n)^2 / \log n$ and conclude that $\Phi_{n, \varepsilon}^\prime (b) < 0$ for sufficiently large $n$. One proves similarly that $\Phi_{n, \varepsilon}^\prime (a) > 0$, so Equation (\ref{proving_heuristic}) holds indeed; this proves Equation (\ref{x_tilde}).

We estimate now the value of $\Phi_{n, \varepsilon} ( x_* ) = \log \phi_{n, \varepsilon} ( x_* )$. We have from Equation (\ref{phi_n}) that
\begin{align*}
\Phi_{n, \varepsilon} ( x_* ) 
= & 
x_* \log k + \log \Gamma(r) - \log \Gamma(r - x_* + 1) + \log \Gamma(n - x_*) - \log \Gamma(n) \\
& + \beta_\varepsilon \sqrt{ \frac{ x_* }{ \log x_* } } + O(\log n).
\end{align*}
Since $x_* \log k + x_* \log r - x_* \log n = 0$,  Lemma \ref{lemma_equatios_615_616} implies 
\begin{equation} \label{value_of_max_1}
\Phi_{n, \varepsilon} ( x_* ) = - \frac{ \lambda x^2_* }{ 2n } +
\beta_\varepsilon \sqrt{ \frac{ x_* }{ \log x_* } } + O( \log n ).
\end{equation}
We note that Equation (\ref{x_tilde}) implies
\begin{equation} \label{value_of_max_2}
\frac{ \lambda x^2_* }{ 2n } =
\frac{1}{2} 
\left( \sqrt[3]{ \frac{3 \beta_\varepsilon^2}{8} } \right)^2 
\frac{ (n/\lambda)^{1/3} }{ \log^{2/3} (n/\lambda) } \big( 1 + o(1) \big)
\end{equation}
and
$$\beta_\varepsilon \sqrt{ \frac{ x_* }{ \log x_* } } \sim
\beta_\varepsilon \left( \sqrt[3]{ \frac{3 \beta_\varepsilon^2}{8} } \right)^{1/2} 
\frac{ (n/\lambda)^{1/3} }{ \log^{1/6} (n/\lambda) } 
\left( \frac{2}{3} \log (n/\lambda) + O( \log \log n ) \right)^{-1/2},$$
that is,
\begin{equation} \label{value_of_max_3}
\beta_\varepsilon \sqrt{ \frac{ x_* }{ \log x_* } } =
\beta_\varepsilon \left( \frac{3}{2} \sqrt[3]{ \frac{3 \beta_\varepsilon^2}{8} } \right)^{1/2} 
\left( \frac{n}{ \lambda } \right)^{1/3} \frac{1}{ \log^{2/3} (n/\lambda) } \big( 1 + o(1) \big).
\end{equation}
Hence, using Equations (\ref{value_of_max_1}), (\ref{value_of_max_2}) and (\ref{value_of_max_3}) and the fact that $k = o(n^{1 - \alpha})$, we obtain
$$\Phi_{n, \varepsilon} ( x_* ) = 
k_\varepsilon \frac{ (n/\lambda)^{1/3} }{ \log^{2/3} (n/\lambda) } \big( 1 + o(1) \big),$$
where, as desired,
\begin{align*}
k_\varepsilon
& =
- \frac{1}{2} \left( \sqrt[3]{ \frac{3 \beta_\varepsilon^2}{8} } \right)^2 + 
\beta_\varepsilon \left( \frac{3}{2} \sqrt[3]{ \frac{3 \beta_\varepsilon^2}{8} } \right)^{1/2} 
=
\left( - \frac{1}{8} + \frac{1}{2} \right) \sqrt[3]{ 3^2 \beta_\varepsilon^4 } =
\frac{ \sqrt[3]{ 3^5 \beta_\varepsilon^4 } }{8}.
\end{align*}
\hfill
\end{proof}

\begin{lemma} \label{lemma_m_tilde_approaches_infinity}
Let $m_* = m_*(n)$ be the integer that maximizes $\mathbb{P}_n^{\{ 0,k \}} ( \Z = m ) M_m$ for $1 \leq m \leq n$. Let $\varepsilon > 0$ and let $x_*$ be as in Proposition \ref{propostion_maximizing_phi}. Then $m_*$ approaches infinity when so does $n$ and, if $k = o(n^{1 - \alpha})$ for some $0 < \alpha < 1$,
$$\max_{1 \leq m \leq n} \mathbb{P}_n^{\{ 0,k \}} ( \Z = m ) M_m \leq \phi_{n,\varepsilon}( x_* ).$$
\end{lemma}
\begin{proof}
First we prove that the integer $m_*(n)$ approaches infinity when so does $n$.
Assume that there exists $K > 0$ and a subsequence $m_*(n_j)$, $j \geq 1$, such that $m_*(n_j) \leq K$ for all $j \geq 1$. It follows that $\mathbb{P}_{n_j}^{ \{0,k\} } ( \Z = m ) M_m$ is bounded for $j \geq 1$. However, it follows from Lemmas \ref{lemma_asymptotic_distribution_Z_0,k} and \ref{lemma_M_m} that, for $m = \left\lfloor (n/\lambda)^{1/2} \right\rfloor$, 
\begin{align*}
\mathbb{P}_n^{ \{0,k\} } ( \Z = m ) M_m 
& = 
(n/\lambda)^{-1/2} \exp \left( -\frac{1}{2} + o(1) + 
\frac{ \beta_0 (n/\lambda)^{1/4} }{ \frac{1}{2} \log^{1/2} (n/\lambda) } (1 + o(1)) \right), \\
& =
\exp \left( - \frac{1}{2} \log (n/\lambda) +
\beta_0 \frac{ (n/\lambda)^{1/4} }{ \frac{1}{2} \log^{1/2} (n/\lambda) } (1 + o(1)) \right),
\end{align*}
and this quantity approaches infinity when so does $n$.
This contradicts the fact that $\mathbb{P}_{n_j}^{ \{0,k\} } ( \Z = m ) M_m$ is bounded for $j \geq 1$, so we have indeed that $m_*(n) \longrightarrow \infty$ as $n \longrightarrow \infty$. 

As a consequence of the first part of the lemma, that we just proved, we have $m_* > N$ for any fixed integer $N$. Thus Lemma \ref{lemma_M_m} and Equation (\ref{phi_n}) imply that, if $\varepsilon > 0$, then $\mathbb{P}_n^{\{ 0,k \}} ( \Z = m_* ) M_{m_*} \leq \phi_{n,\varepsilon}(m_*)$ holds for sufficiently large $n$. The result follows from the definition of $x_*$ in Proposition \ref{propostion_maximizing_phi}.
\hfill
\end{proof}

\begin{theorem} \label{theorem_T}
Let $k = k(r)$ and $n = n(r)$ be sequences such that $n = kr$ and, for some $0 < \alpha < 1$, $k = o(n^{1 - \alpha})$ as $r$ approaches infinity. 
Let $\mathbb{E}_n^{\{ 0,k \}} (\T)$ be the expected value of $\T$ over the class of mappings on $n$ nodes with indegrees restricted to the set $\{0,k\}$. Then,
$$\log \mathbb{E}_n^{\{ 0,k \}} (\T) = 
k_0 \frac{ (n/\lambda)^{1/3} }{ \log^{2/3} (n/\lambda) } (1 + o(1)),$$
as $r$ approaches infinity, where $\lambda = k-1$, $k_0 = \frac{3}{2} (3I)^{2/3}$ and $I$ is given in Equation (\ref{I_definition_constant}).
\end{theorem}
\begin{proof}
It follows from Equation (\ref{expected_value_of_T}) that, if $1 \leq m_0 \leq n$ and $m_*$ is the integer that maximizes $\mathbb{P}_n^{\{ 0,k \}} ( \Z = m ) M_m$ for $1 \leq m \leq n$, then 
$$\mathbb{P}_n^{\{ 0,k \}} ( \Z = m_0 ) M_{m_0} \leq \mathbb{E}_n^{\{ 0,k \}} ( \T ) \leq n \mathbb{P}_n^{\{ 0,k \}} ( \Z = m_* ) M_{m_*}.$$
Let $\varepsilon \in (-1,0)$.
Since $m_0 = \lfloor t_* \rfloor$ implies $m_0 = t_* \big( 1 + o(1) \big)$, it follows from Lemma \ref{lemma_M_m}, Equation (\ref{distribution_Z_gamma_function}) and Proposition \ref{propostion_maximizing_phi} that
\begin{equation} \label{final_lower_bound_T}
\log \mathbb{E}_n^{\{ 0,k \}} (\T) \geq 
k_\varepsilon \frac{ (n/\lambda)^{1/3} }{ \log^{2/3} (n/\lambda) } (1 + o(1)).
\end{equation}
On the other hand, if $\varepsilon \in (0,1)$, then Lemma \ref{lemma_m_tilde_approaches_infinity} implies that $\mathbb{E}_n^{\{ 0,k \}} ( \T ) \leq n \phi_{n,\varepsilon}(x_*)$ and thus, by Proposition \ref{propostion_maximizing_phi},
\begin{equation} \label{final_upper_bound_T}
\log \mathbb{E}_n^{\{ 0,k \}} ( \T ) \leq 
k_\varepsilon \left( \frac{n}{ \lambda } \right)^{1/3} \frac{1}{ \log^{2/3} n } (1 + o(1)).
\end{equation}

Let $\varepsilon_1 > 0$. Since $k_\varepsilon \to k_0$ as $\varepsilon \to 0$, we have $(1 - \varepsilon_1) k_0 < k_\varepsilon < (1 + \varepsilon_1) k_0$ for sufficiently small $\varepsilon$. The result follows from Equations (\ref{final_lower_bound_T}) and (\ref{final_upper_bound_T}) and Lemma \ref{lemma_calc}.
\hfill
\end{proof}

\begin{corollary} \label{corollary_T_fixed_k}
Let $k \geq 2$ be a fixed integer and let $\mathbb{E}_n^{\{ 0,k \}} (\T)$ be the expected value of $\T$ over the class of mappings on $n$ nodes with indegrees restricted to the set $\{0,k\}$. Then,
$$\log \mathbb{E}_n^{\{ 0,k \}} (\T) = 
k_0 \frac{ (n/\lambda)^{1/3} }{ \log^{2/3} (n/\lambda) } (1 + o(1)),$$
as $n$ approaches infinity, where $\lambda = k-1$, $k_0 = \frac{3}{2} (3I)^{2/3}$ and $I$ is given in Equation (\ref{I_definition_constant}).
\end{corollary}

\section{Expected Value of $\BB$} \label{section_expected_value_of_B}

We obtain asymptotic estimates for the expectation of $\BB$ over $\{0,k\}$-mappings using arguments similar to those in Section \ref{section_expected_value_of_T}.
Let $\mu_m$ denote the expected value of the product of the cycle lengths of a uniform random permutation of $S_m$. Using Lemma \ref{lemma_restriction_to_Z} we can write the expected value of $\BB$ over all $\{0,k\}$-mappings on $n = kr$ nodes as
\begin{equation} \label{expected_value_of_B}
\mathbb{E}_n^{ \{0,k\} } ( \BB ) = \sum_{m = 1}^n \mathbb{P}_n^{ \{0,k\} } ( \Z = m ) \mu_m.
\end{equation}
The following lemma gives an asymptotic estimate for $\mu_m$ as $m$ approaches infinity.

\begin{lemma} \label{lemma_mu_m}
Let $\mu_m$ be the expected value of the product of the cycle lengths of a random uniform permutation of $S_m$. Then, as $m$ approaches infinity,
$$\mu_m = \frac{ \exp( 2 \sqrt{m} ) }{ 2 \sqrt{ \pi e} m^{3/4} } (1 + o(1)).$$
In particular, for any $\varepsilon > 0$, there exists $N_\varepsilon$ such that, for all $m > N_\varepsilon$,
$$(2 - \varepsilon) \sqrt{m} < \log \mu_m < (2 + \varepsilon) \sqrt{m}.$$
\end{lemma}

It is clear from Equation (\ref{expected_value_of_B}) that, if $m_*$ is the integer that maximizes $\mathbb{P}_n^{ \{0,k\} } (\Z = m) \mu_m$ for $1 \leq m \leq n$ and $m_0$ is an integer in $[1,n]$, then
$$\mathbb{P}_n^{ \{0,k\} } ( \Z = m_0 ) \mu_{ m_0 } \leq \mathbb{E}_n^{ \{0,k\} } (\BB) \leq 
n \mathbb{P}_n^{ \{0,k\} } ( \Z = m_* ) \mu_{ m_* }.$$
Let $\varepsilon > 0$. It follows from Lemma \ref{lemma_mu_m} that, for sufficiently large $n$,
\begin{equation} \label{bounds_Z}
\mathbb{P}_n^{ \{0,k\} } ( \Z = m_0 ) e^{ (2 - \varepsilon) \sqrt{m_0} } \leq \mathbb{E}_n^{ \{0,k\} } (\BB) \leq 
n \mathbb{P}_n^{ \{0,k\} } ( \Z = m_* ) e^{ (2 + \varepsilon) \sqrt{m_*} },
\end{equation}
provided that $m^*$ approaches infinity when so does $n$; we defer the proof of this claim to the proof of Theorem \ref{theorem_BB_0,k}.

In light of Equation (\ref{distribution_Z_gamma_function}), in order to obtain upper and lower bounds by Equation (\ref{bounds_Z}) we consider the function
\begin{equation}
U_{n, \varepsilon}(x) = \lambda x k^{x - 1} \frac{ \Gamma(r) }{ \Gamma(r - x + 1) } \frac{ \Gamma(n - x) }{ \Gamma(n) } 
\exp( (2 + \varepsilon) \sqrt{x} ),
\label{U_definition}
\end{equation}
where $\varepsilon$ denotes a real number that may be positive or negative.
If $x_*$ is the point that maximizes $U_n(x)$ for $x \in (0,n)$ and $m_0$ is an integer in $[1,n]$, then Equation (\ref{bounds_Z}) implies, for sufficiently large $n$, that
\begin{equation}
U_{n, -\varepsilon}(m_0) \leq \mathbb{E}_n^{ \{0,k\} } (\BB) \leq 
n \cdot U_{n, \varepsilon}(x_*).
\label{bounds_Z_U}
\end{equation}
In order to simplify the next calculations, we consider $H_{n, \varepsilon}(x) = \log U_{n, \varepsilon}(x)$ and note that $x_*$ is a local maximum of $U_{n, \varepsilon}(x)$ if and only if it is a local maximum of $H_{n, \varepsilon}(x)$. It is known that the inverse of the Gamma function has simple zeroes in the non-positive integers, so the function $H_{n, \varepsilon}(x)$ is not well defined for $x \in \{ r+1, r+2, \dots \}$. 
We consider the range $[1,r]$ and note that this is not an issue because $\mathbb{P}_n(\Z = m) = 0$ for $m > r$. Indeed, every cyclic node in a $\{0,k\}$-mapping $\varphi: [n] \longrightarrow [n]$ has indegree $k$, hence $n$ and $m$ must satisfy $k \cdot m \leq n$.

We recall that a real function is differentiable in this range if and only if it is differentiable in $(1-\delta, h+\delta)$ for some $\delta > 0$.

\begin{proposition}
For each $r \geq 1$, there exists a unique point $x_*$ that maximizes the function $H_{n, \varepsilon}(x)$ for $x \in [1,r]$, where $n = rk$.
Moreover, $H_{n, \varepsilon}(x_*)$ and $H_{n, \varepsilon}( \lfloor x_* \rfloor )$ are both given by
$$\left( 1 + \frac{\varepsilon}{2} \right)^{4/3} \frac{3}{2} \left( \frac{ n }{ \lambda } \right)^{1/3} (1 + o(1)),$$
where $\lambda = k-1$.
\label{propostion_maximizing_H}
\end{proposition}
\begin{proof}
The function $\log \Gamma(x)$ is infinitely differentiable for 
$x > 0$; see \cite{AbramowitzStegun1965HandbookMathFunctions}. Hence, for each $n \geq 1$, $H_{n, \varepsilon}(x)$ is infinitely differentiable for $x \in [1,r]$. It follows from Lemma \ref{lemma_gamma_function} that, for $x \in [1,r]$, 
\begin{equation}
\begin{split}
H_{n, \varepsilon}^\prime (x) 
& = 
\frac{1}{x} + \log k + \frac{d}{dx} \log \left( \frac{ \Gamma(n - x) }{ \Gamma(r - x + 1) } \right) 
+ \left( 1 + \frac{\varepsilon}{2} \right) x^{-1/2} \\
& = 
\log k + \frac{1}{x} + \left( 1 + \frac{\varepsilon}{2} \right) x^{-1/2} + \Psi(r - x + 1) - \Psi(n - x).
\end{split}
\label{derivative_H}
\end{equation}
Using Lemma \ref{lemma_gamma_function} we obtain
\begin{align*}
H_{n, \varepsilon}^\prime (1) 
& = \log k + 2 + \frac{\varepsilon}{2} + \Psi(r) - \Psi(n-1) \\
& = 
\log k + 2 + \frac{\varepsilon}{2} + \log r + O \left( \frac{1}{r} \right) - \log (n-1) + O \left( \frac{1}{n-1} \right) \\
& =
\log k + 2 + \frac{\varepsilon}{2} + \log r - \log n - \log \left( 1 - \frac{1}{n} \right) + O \left( \frac{1}{r} \right) \\
& =
\log k + 2 + \frac{\varepsilon}{2} + \log r - \log k - \log r + O \left( \frac{1}{n} \right) = 
2 + \frac{\varepsilon}{2} + O \left( \frac{1}{r} \right).
\end{align*}
Therefore $H_{n, \varepsilon}^\prime (1) > 0$ for sufficiently large values of $n$. 
On the other hand,
\begin{align*}
H_{n, \varepsilon}^\prime (r) 
& =
\log k + \frac{1}{r} + \left( 1 + \frac{\varepsilon}{2} \right) r^{-1/2} + \Psi(1) - \Psi(n - r) \\
& =
\log k + \frac{1}{r} + \left( 1 + \frac{\varepsilon}{2} \right) r^{-1/2} + \Psi(1) - \log (n-r) + O \left( \frac{1}{n-r} \right).
\end{align*}
Since 
$$\frac{1}{n-r} = \frac{1}{k-1} \cdot \frac{1}{r} = O \left( \frac{1}{r} \right),$$
it follows that
$$H_{n, \varepsilon}^\prime (r) = 
\log k + \Psi(1) - \log (k-1) - \log r + O \left( \frac{1}{ r^{1/2} } \right)
= - \log r + O(1).$$
Hence $n = kr$ implies that $H_{n, \varepsilon}^\prime (r) < 0$ for sufficiently large $n$. This proves the existence of a point $x_*$ that is a local maximum of $H_{n, \varepsilon} (x)$.
Also,
$$H_{n, \varepsilon}^{ \prime \prime }(x) = - x^{-2} - \left( \frac{1}{2}+ \frac{\varepsilon}{4} \right) x^{-3/2} - \Psi^\prime (r - x + 1) + \Psi^\prime (n - x).$$
Since $r - x + 1 < n - x$, it follows from Lemma \ref{lemma_gamma_function} that $\Psi^\prime (n - x) < \Psi^\prime (r - x + 1)$ and thus $H_{n, \varepsilon}^{ \prime \prime }(x) < 0$ for $x \in [1,r]$. This proves that $x_*$ is unique.

We obtain next a heuristic estimate for $x_*$ as $n$ approaches infinity. Using Equation (\ref{derivative_H}) and Lemma \ref{lemma_equatios_615_616} one proves that, for $x = o(r)$,
\begin{equation}
\begin{split}
H_{n, \varepsilon}^\prime (x)
& =
\frac{1}{x} + \left( 1 + \frac{\varepsilon}{2} \right) x^{-1/2} - \frac{x}{r} + \frac{x}{n} + O \left( \frac{x^2}{r^2} \right) + O \left( \frac{1}{r} \right) \\
& =
\frac{1}{x} + \left( 1 + \frac{\varepsilon}{2} \right) x^{-1/2} - \frac{ (k-1) x}{n} + O \left( \frac{x^2}{r^2} \right) + O \left( \frac{1}{r} \right).
\end{split}
\label{derivative_H_2}
\end{equation}
%
%
We recall that $\lambda = k-1$ and consider the equation
$$\frac{1}{x} + \left( 1 + \frac{\varepsilon}{2} \right) x^{-1/2} - \frac{ \lambda x}{n} = 0,$$
that is,
$$\left( 1 + \frac{\varepsilon}{2} \right) x^{-1/2} \left( 1 + O \left( \frac{1}{ \sqrt{x} } \right) \right) = \frac{ \lambda x}{n}.$$
The equation above suggests that 
\begin{equation}
x_* = \left( \left( 1 + \frac{\varepsilon}{2} \right) \frac{n}{\lambda} \right)^{2/3} \big( 1 + o(1) \big).
\label{estimate_maximum_H}
\end{equation}
In order to confirm that Equation (\ref{estimate_maximum_H}) holds, we prove that 
\begin{equation}
H_{n, \varepsilon}^\prime \left( \left[ \left( 1 + \frac{\varepsilon}{2} \right) \frac{n}{\lambda} \right]^{2/3} (1 + \delta_n ) \right) < 0 
\label{confirming_estimate_1}
\end{equation}
and
\begin{equation}
H_{n, \varepsilon}^\prime \left( \left[ \left( 1 + \frac{\varepsilon}{2} \right) \frac{n}{\lambda} \right]^{2/3} (1 - \delta_n ) \right) >0,
\label{confirming_estimate_1}
\end{equation}
for some small $\delta_n = o(1)$ to be determined. We observe that Equation (\ref{derivative_H_2}) implies
\begin{align*}
& \hspace{15pt} H_{n, \varepsilon}^\prime \left( \left[ \left( 1 + \frac{\varepsilon}{2} \right) \frac{n}{\lambda} \right]^{2/3} (1 + \delta_n ) \right) \\
& =
\left( \left( 1 + \frac{\varepsilon}{2} \right) \frac{n}{\lambda} \right)^{-2/3} (1 + \delta_n )^{-1} + \left( 1 + \frac{\varepsilon}{2} \right) \left( \left( 1 + \frac{\varepsilon}{2} \right) \frac{n}{\lambda} \right)^{-1/3} (1 + \delta_n )^{-1/2} \\
& \hspace{10pt} 
- \frac{ (k-1) }{n} \left( \left( 1 + \frac{\varepsilon}{2} \right) \frac{n}{\lambda} \right)^{2/3} (1 + \delta_n ) + 
O \big( r^{-2/3} \big) + O \big( r^{-1} \big) \\
& =
\left( 1 + \frac{\varepsilon}{2} \right)^{2/3} \left( \frac{n}{\lambda} \right)^{-1/3} 
(1 + \delta_n )^{-1/2}
- \left( 1 + \frac{\varepsilon}{2} \right)^{2/3} \left( \frac{n}{\lambda} \right)^{-1/3} (1 + \delta_n ) \\
& 
\hspace{10pt}+ O \big( r^{-2/3} \big) \\
& =
\left( 1 + \frac{\varepsilon}{2} \right)^{2/3} \left( \frac{n}{\lambda} \right)^{-1/3} \left[ (1 + \delta_n )^{-1/2} - (1 + \delta_n ) + O \big( r^{-1/3} \big) \right] \\
& =
\left( 1 + \frac{\varepsilon}{2} \right)^{2/3} \left( \frac{n}{\lambda} \right)^{-1/3} \left[ \left( 1 - \frac{1}{2} \delta_n + O( \delta_n^2 ) \right) - (1 + \delta_n ) + O \big( r^{-1/3} \big) \right] \\
& =
\left( 1 + \frac{\varepsilon}{2} \right)^{2/3} \left( \frac{n}{\lambda} \right)^{-1/3} \left( - \frac{3}{2} \delta_n + O( \delta_n^2 ) + O \big( r^{-1/3} \big) \right).
\end{align*}

It is of our interest to write 
$$H_{n, \varepsilon}^\prime \left( \left[ \left( 1 + \frac{\varepsilon}{2} \right) \frac{n}{\lambda} \right]^{2/3} (1 + \delta_n ) \right) =
\left( 1 + \frac{\varepsilon}{2} \right)^{2/3} \left( \frac{n}{\lambda} \right)^{-1/3} \left( - \frac{3}{2} \delta_n + o( \delta_n ) \right),$$
as this would allow us to determine if the left-hand side of the equation above is positive or negative, depending on the value of $\delta_n$. We set $\delta_n = r^{-1/4}$ and conclude that 
\begin{equation} \label{derivative_H_proving_maximum_1}
H_{n, \varepsilon}^\prime \left( \left[ \left( 1 + \frac{\varepsilon}{2} \right) \frac{n}{\lambda} \right]^{2/3} (1 + \delta_n ) \right) <0,
\end{equation}
for sufficiently large $n$. One proves similarly that, for sufficiently large $n$, we have 
$$H_{n, \varepsilon}^\prime \left( \left[ \left( 1 + \frac{\varepsilon}{2} \right) \frac{n}{\lambda} \right]^{2/3} (1 - \delta_n ) \right) =
\left( 1 + \frac{\varepsilon}{2} \right)^{2/3} \left( \frac{n}{\lambda} \right)^{-1/3} \left( \frac{3}{2} \delta_n + o( \delta_n ) \right).$$
Hence, for sufficiently large $n$,
\begin{equation} \label{derivative_H_proving_maximum_2}
H_{n, \varepsilon}^\prime \left( \left[ \left( 1 + \frac{\varepsilon}{2} \right) \frac{n}{\lambda} \right]^{2/3} (1 - \delta_n ) \right) >0.
\end{equation}
%
Equations (\ref{derivative_H_proving_maximum_1}) and (\ref{derivative_H_proving_maximum_2}) imply that Equation (\ref{estimate_maximum_H}) holds indeed.

We estimate the value of $H_{n, \varepsilon}( x_*)$. We recall that $H_{n, \varepsilon}(x_*) = \log U_{n, \varepsilon}(x_*)$, where $U_{n, \varepsilon}(x_*)$ is defined in Equation (\ref{U_definition}). 
It follows from Lemma \ref{lemma_equatios_615_616} that
\begin{align*}
H_{n, \varepsilon}(x_*) 
& = 
x_* \log k - \frac{ x_*^2 }{ 2h } + x_* \log r + \frac{ x_*^2 }{ 2n } - x_* \log n + (2 + \varepsilon) \sqrt{x_*} + O( \log r ) \\
& = 
- (k-1) \frac{ x_*^2 }{ 2n } + (2 + \varepsilon) \sqrt{x_*} + O( \log r ).
\end{align*}
Hence, by Equation (\ref{estimate_maximum_H}),
\begin{align*}
H_{n, \varepsilon}(x_*) 
& = 
\left( \frac{n}{\lambda} \right)^{1/3} \left[ - \frac{1}{2} \left( \frac{2 + \varepsilon}{2} \right)^{4/3} 
+ (2 + \varepsilon) \left( \frac{2 + \varepsilon}{2} \right)^{1/3} \right] \big( 1 + o(1) \big) \\
& = 
\left( 1 + \frac{\varepsilon}{2} \right)^{4/3} \left( \frac{n}{\lambda} \right)^{1/3} 
\left[ - \frac{1}{2} + 2 \right] \big( 1 + o(1) \big) \\
& = 
\left( 1 + \frac{\varepsilon}{2} \right)^{4/3} \frac{3}{2} \left( \frac{n}{\lambda} \right)^{1/3} \big( 1 + o(1) \big),
\end{align*}
as desired. The estimate of $H_{n, \varepsilon}( \lfloor x_* \rfloor )$ follows easily from the fact that 
$$\lfloor x_* \rfloor = x_* - \{ x_* \} = x_* + O(1) = x_* \big( 1 + o(1) \big).$$
\hfill
\end{proof}

\begin{theorem} \label{theorem_BB_0,k}
Let $k = k(r)$ and $n = n(r)$ be sequences such that $n = kr$ and, for some $0 < \alpha < 1$, $k = o(n^{1 - \alpha})$ as $r$ approaches infinity. 
Let $\mathbb{E}_n^{\{ 0,k \}} (\BB)$ be the expected value of $\BB$ over the class of mappings on $n$ nodes with indegrees restricted to the set $\{0,k\}$. Then,
$$\log \mathbb{E}_n^{ \{0,k\} } (\BB) = \frac{3}{2} \left( \frac{ n }{ \lambda } \right)^{1/3} (1 + o(1)),$$
where $\lambda = k-1$.
\end{theorem}
\begin{proof}
We recall that the bounds in Equation (\ref{bounds_Z}) hold provided that the integer $m_* = m_*(n)$ that maximizes $\mathbb{P}_n^{ \{0,k\} } (\Z = m) \mu_m$ for $1 \leq m \leq n$ tends to infinity when so does $n$. We prove this claim next. Indeed, if there exists $C > 0$ and a subsequence $( m_*(n_j) )_j$ such that $m_*(n_j) \leq C$ for all $j \geq 1$, then $\mathbb{P}_{n_j}^{ \{0,k\} } (\Z = m) \mu_m$ is bounded for $j \geq 1$. However, it follows from Lemma \ref{lemma_asymptotic_distribution_Z_0,k} that, for $m = \left\lfloor n^{1/2} \right\rfloor$, 
$$\mathbb{P}_n^{ \{0,k\} } (\Z = m) \mu_m \sim 
\frac{ \lambda n^{1/2} }{ n 2 \sqrt{\pi e} n^{3/4} } 
\exp \left( - \frac{ \lambda ( n^{1/2} )^2 }{2n} + 2 ( n^{1/2} )^{1/2} \right),$$
hence,
$$\mathbb{P}_n^{ \{0,k\} } (\Z = m) \mu_m \sim 
\left( \frac{ \lambda^2 e^{k} }{ 4 \pi } \right)^{1/2} \frac{ e^{ 2 n^{1/4} } }{ n^{5/4} }.$$
Thus, for $m = \left\lfloor n^{1/2} \right\rfloor$, $\mathbb{P}_n^{ \{0,k\} } (\Z = m) \mu_m$ approaches infinity when so does $n$. This is a contradiction, so we have indeed that $m_*(n) \longrightarrow \infty$ as $n \longrightarrow \infty$. 

Let $h = n/k$. We recall that $\mathbb{P}_n(\Z = m) = 0$ for $m > h$. It follows from Equation (\ref{bounds_Z_U}) that
$$\max_{1 \leq m \leq n} \mathbb{P}_n(\Z = m) \mu_m = 
\max_{1 \leq m \leq h} \mathbb{P}_n(\Z = m) \mu_m \leq
n \cdot \max_{1 \leq x \leq h} U_{n, \varepsilon}(x) = n \cdot U_{n, \varepsilon}(x_*).$$
Since $n = \exp(\log n)$, using Proposition \ref{propostion_maximizing_H} we conclude that
\begin{equation}
\log \mathbb{E}_n^{ \{0,k\} } (\BB) \leq
\left( 1 + \frac{\varepsilon}{2} \right)^{4/3} \frac{3}{2} \left( \frac{ n }{ \lambda } \right)^{1/3} (1 + o(1)).
\label{near_final_upper_bound_BB}
\end{equation}

If $\varepsilon > 0$ and $m_0 = \lfloor x_* \rfloor$, then Equation (\ref{bounds_Z_U}) and Proposition \ref{propostion_maximizing_H} imply
\begin{equation}
\log \mathbb{E}_n^{ \{0,k\} } (\BB) \geq  H_{n, -\varepsilon}(m_0) = 
\left( 1 - \frac{\varepsilon}{2} \right)^{4/3} \frac{3}{2} \left( \frac{ n }{ \lambda } \right)^{1/3} (1 + o(1)).
\label{final_lower_bound_BB}
\end{equation}
Let $\varepsilon_1 > 0$
Since $\left( 1 + \varepsilon/2 \right)^{4/3} \to 1$ as $\varepsilon \to 0$, we have $1 - \varepsilon_1 < \left( 1 - \varepsilon/2 \right)^{4/3}$ and $\left( 1 + \varepsilon/2 \right)^{4/3} < 1 + \varepsilon_1$ for sufficiently small $\varepsilon$. The result follows from Equations (\ref{final_lower_bound_T}) and (\ref{final_upper_bound_T}) and Lemma \ref{lemma_calc}.

\hfill
\end{proof}

\begin{corollary} \label{corollary_B_fixed_k}
Let $k \geq 2$ be a fixed integer and let $\mathbb{E}_n^{\{ 0,k \}} (\BB)$ be the expected value of $\BB$ over the class of mappings on $n$ nodes with indegrees restricted to the set $\{0,k\}$. Then, as $n$ approaches infinity, 
$$\log \mathbb{E}_n^{\{ 0,k \}} (\BB) = \frac{3}{2} \left( \frac{ n }{ \lambda } \right)^{1/3} (1 + o(1)),$$
where $\lambda = k-1$.
\end{corollary}

\section{Lognormality} \label{section_lognormality}

Let $\mu_n^* = \frac{1}{2} \log^2 (\sqrt{n})$, $\sigma_n^* = \frac{1}{ \sqrt{3} } \log^{3/2} (\sqrt{n})$ and $\mu_n = \frac{1}{2} \log^2 (\sqrt{ \frac{n}{\lambda} })$, $\sigma_n = \frac{1}{ \sqrt{3} } \log^{3/2} (\sqrt{ \frac{n}{\lambda} })$.
Harris proved that the sequence of random variables defined over the space of random mappings on $n$ nodes as $X_n = (\log \T - \mu_n^*) / \sigma_n^*$, $n \geq 1$, converges weakly to a standard normal distribution \cite{harris1973DistributionMappings}. In this section we prove an analogue of this result for $\{ 0,k\}$-mappings:
\begin{equation} \label{convergence_distribution_0k_mappings}
\lim_{n \rightarrow \infty} \mathbb{P}_n^{\{0,k\}} 
\left( \frac{ \log \T - \mu_n }{\sigma_n} \leq x \right) = 
\frac{1}{ \sqrt{2\pi} } \int_{-\infty}^x e^{-t^2/2} dt.
\end{equation}
The analogous result for the parameter $\BB$ is proved from Equation (\ref{convergence_distribution_0k_mappings}) by showing that the random variable $\chi_n = \log \BB - \log \T$, when properly normalized, converges in probability to zero.

We write the probability in Equation (\ref{convergence_distribution_0k_mappings}) using the law of total probability: we partition the space of $\{0,k\}$-mappings according to the values $m \in [1,r]$ that $\Z$ assumes.
Let $\xi_1, \xi_2$ be as in Lemma \ref{lemma_concentration}. We partition the interval $[1,r]$ into three subintervals:
\begin{itemize}
\item $ I_1=\{ m: 1\leq m < \xi_1 \}$, 
\item $I_2=\{ m: \xi_1\leq  m \leq  \xi_2 \}$, 
\item $I_3=\{ m:  \xi_2 < m\leq r \}$.
\end{itemize}
Then, by the law of total probability,  
\begin{equation} \label{splitting_probability}
\mathbb{P}_n^{\{ 0,k\}} \left( \log \T \leq \mu_n + x \sigma_n \right) = 
\zeta_1 + \zeta_2 + \zeta_3,
\end{equation}
where 
\begin{equation} \label{zetaj}
\zeta_j = \sum_{m \in I_j} \mathbb{P}_n^{\{ 0,k\}} (\Z = m)
\mathbb{P}_n^{\{ 0,k\}} (\log \T \leq \mu_n + x \sigma_n | \Z = m).
\end{equation}
The conditional probabilities $\mathbb{P}_n^{\{ 0,k\}} (\log \T \leq \mu_n + x \sigma_n | \Z = m)$ in Equation (\ref{zetaj}) can be trivially bounded by $1$, so Lemma \ref{lemma_concentration} clearly implies that
\begin{equation} \label{zeta13small}
\zeta_1 = O\left(  \exp \left( - c \log^{1/4} \left( \frac{n}{\lambda }\right) \right)\right)   
 \quad \text{and} 
\quad \zeta_3= O\left(  \exp \left( - c \log^{1/4} \left( \frac{n}{\lambda }\right) \right)\right),    
\end{equation} 
where $c$ is any positive constant less than $2^{3/4}$. Our 
estimates for $\zeta_2$, the asymptotic main term in (\ref{splitting_probability}),  
  use a strong version of the well known fact that the order of a random permutation is asymptotically log-normal
\cite{BarbourTavare1994LognormalityRate,ErdosTuran1967lognormality}.   Denote by ${\bf Q}_m$ the uniform probability measure on the symmetric group $S_m$ and by $\phi(x) = \frac{1}{2\pi} \int_{-\infty}^x e^{-t^2/2} dt$ the standard normal distribution. 

\begin{theorem}[Barbour and Tavar\'e \cite{BarbourTavare1994LognormalityRate}] \label{theorem_BTcenter}
Let $\alpha_m = \frac{1}{2}\log^{2} m + \log m \log\log m$ and $\beta_m = \frac{1}{\sqrt{3}}\log^{3/2} m$. Then, there exists a constant $K>0$ such that, for all real numbers $x$ and all  integers $m>1$,
$$\biggl | {\bf Q}_m\left( \log \T \leq \alpha_m + x \beta_m \right) - \phi(x) \biggr| \leq
\frac{K}{\sqrt{\log m}}.$$
\end{theorem}

In the next lemma, we use Lemma \ref{lemma_restriction_to_Z} and Theorem \ref{theorem_BTcenter} to approximate all the conditional probabilities in the definition of $\zeta_2$. It turns out that the interval $[\xi_1,\xi_2]$ is narrow enough so that all conditional probabilities in the summand of $\zeta_2$ are approximately $\phi(x)$.

\begin{lemma} \label{lemma_Pn_conditional}
For $n = kr$ and $m \in I_2$, let 
$$\delta_x(m,n) = \mathbb{P}_n^{\{ 0,k\}} 
\left( \log \T \leq \mu_n + x\sigma_n | \Z = m \right) - \phi(x),$$
and let $\Delta_x(n) = \max \{ |\delta_x(m,n)|, m \in I_2 \}$. There are positive constants $K_{1},K_{2}$ such that
\[ \Delta_x(n) = \max_{m \in I_2} |\delta_x(m,n)| \leq 
\frac{K_1}{\log^{1/4}(\sqrt{n/\lambda})}+\frac{K_2|x|}{\log^{3/4}(\sqrt{n/\lambda})}.\]
\end{lemma}
\begin{proof}
It follows from Lemma \ref{lemma_restriction_to_Z} that, for any $m \in I_2$,
\begin{equation} \label{reduce2perms}
\mathbb{P}_n^{\{ 0,k\}} \left( \frac{ \log \T - \mu_n}{\sigma_n} \leq x \Bigr| \Z = m \right)
={\bf Q}_m\left( \frac{\log \T - \mu_n}{\sigma_n} \leq x \right).
\end{equation}
Let $\alpha_m$ and $\beta_m$ be as in Theorem \ref{theorem_BTcenter}. Define $y=y(n,m,x)$ to be the real number for which $\mu_n + x \sigma_n = \alpha_m + y\beta_m$. Then,
\begin{equation} \label{yonly}
{\bf Q}_m \left( \frac{\log \T - \mu_n}{\sigma_n} \leq x \right) = 
{\bf Q}_m \left( \frac{\log \T - \alpha_m}{\beta_m} \leq y \right).
\end{equation}
Therefore, by Equations (\ref{reduce2perms}) and (\ref{yonly}),
$$|\delta_x(m,n)| =
\left| {\bf Q}_m \left( \frac{\log \T -\alpha_m}{\beta_m} \leq y \right) - \phi(x) \right|,$$
and thus, by the triangle inequality,
\begin{equation} \label{triangle}
|\delta_x(m,n)| \leq
\left|{\bf Q}_m \left( \frac{\log \T - \alpha_m}{\beta_m} \leq y \right) - \phi(y) \right|
+ \left| \phi(y) - \phi(x) \right|.
\end{equation}
We note that Theorem \ref{theorem_BTcenter} implies that, for some constant $K_3 > 0$,
\begin{equation}
\label{postapp}
\left| {\bf Q}_m \left( \frac{\log \T -\alpha_m}{\beta_m} \leq y \right) - \phi(y )\right| \leq 
\frac{K_3}{\sqrt{\log m}}.
\end{equation}
Also, from the definition of $\phi$ we obtain 
\begin{equation} \label{phy_y_phi_x}
|\phi(y)-\phi(x)| = \left| \int_x^y \frac{e^{-t^{2}/2}}{\sqrt{2\pi}} \, dt \right| \leq |y-x|.
\end{equation}
Combining Equations (\ref{triangle})-(\ref{phy_y_phi_x}) we obtain
\begin{equation}
\label{back2P}
|\delta_x(m,n)| \leq \frac{K_3}{\sqrt{\log m}} + |y-x|.
\end{equation}
In order to estimate $|y-x|$, we note that the definition of $y$ implies
\begin{equation} \label{y-x}
y - x = \frac{ (\mu_n-\alpha_m)+x(\sigma_n-\beta_m) }{\beta_m}.
\end{equation}
Since $I_2 = \{ m:  m_{\#}^{1-\varepsilon_n} \leq  m \leq  m_{\#}^{1+\varepsilon_n} \}$, where $m_\#=\sqrt{n/\lambda}+O(1)$ and $\varepsilon_n = \log^{-3/4} (\sqrt{n/\lambda})$, we have for $m\in I_2$ that
$$\log m = \log\left(\sqrt{n/\lambda} \right) (1+O(\varepsilon_n)).$$
Combining this with $\beta_m= \frac{1}{\sqrt{3}}\log^{3/2} m$ and 
$\sigma_n = \frac{1}{\sqrt{3}}\log^{3/2} \sqrt{n/\lambda}$, we obtain $\beta_m = \sigma_n( 1 +O(\varepsilon_n))$, hence
\begin{equation} \label{betam}
\sigma_n - \beta_m = O(\sigma_n \varepsilon_n) = O(\beta_m \varepsilon_n).
\end{equation}
Using the same argument we prove that $\mu_n = \alpha_m ( 1 +O(\varepsilon_n))$, and thus
\begin{equation} \label{alpham}
\alpha_m - \mu_n = O(\alpha_m \varepsilon_n) = O \left( \beta_m \log^{-1/4} \left( \sqrt{n/\lambda} \right) \right).
\end{equation}
It follows from Equations (\ref{y-x})-(\ref{alpham}) that
\begin{equation} \label{y-x_2}
y-x = O \left( \log^{-1/4} \left( \sqrt{n/\lambda} \right) \right) +  O(|x|\varepsilon_n).
\end{equation}
Since $m > \xi_1 = O \left( \log \sqrt{n/\lambda} \right)$, it follows from Equations (\ref{back2P}) and (\ref{y-x_2}) that 
$$|\delta_x(m,n)| = 
O \left( \log^{-1/2} \left( \sqrt{n/\lambda} \right) \right) + O \left( \log^{-1/4} \left( \sqrt{n/\lambda} \right) \right) +  
O\left(|x| \log^{-3/4} (n)\right).$$
We note that the right-hand side of the equation above depends on $n$ and $x$, but not on $m$.
The result follows at once from the estimate
\begin{equation} \label{xdepbound}
\Delta_x(n) = \max_{m \in I_2} |\delta_x(m,n)| \leq 
\frac{K_1}{\log^{1/4}(\sqrt{n/\lambda})}+\frac{K_2|x|}{\log^{3/4}(\sqrt{n/\lambda})},
\end{equation}
where $K_1, K_2$ are positive constants. 
\hfill
\end{proof}

\begin{theorem} \label{theorem_lognormalityT}
Let $k = k(r)$ and $n = n(r)$ be sequences such that $n = kr$ and, for some $0 < \alpha < 1$, $k = o(n^{1 - \alpha})$ as $r$ approaches infinity. 
Let $\mu_n=\frac{1}{2}\log^{2}(\sqrt{n/\lambda})$, $\sigma_n^{2} = \frac{1}{3}\log^{3}(\sqrt{n/\lambda})$. Let $\T(f)$ denote the least common multiple of the length of the cycles of a mapping $f$ and, for $r \geq 1$, let $X_n$ be the random variable defined over the space of $\{0,k\}$-mappings on $n$ nodes as $X_n = (\log \T - \mu_n) / \sigma_n$. 
Then, the sequence defined by $X_n$ converges in distribution to a standard normal distribution. Furthermore
\begin{equation}
\label{approximation}
\mathbb{P}_n^{\{0,k\}} \left( \log \T \leq \mu_n + x \sigma_n \right) = \phi(x) +  O\left(\frac{|x|+\log^{1/2}(\sqrt{n/\lambda})}{  \log^{3/4}(\sqrt{n/\lambda})}\right). \end{equation}
\end{theorem}
\vskip0cm\noindent
{\bf Remark.}   
Because the distribution function of  $X_n$ converges pointwise to $\phi(x)$,  we know from  Lemma 3 in Section 8.2 of  \cite{ChowTeicher}
that $X_{n}$  must also converge uniformly;  there is a function $\omega(n)$ such that  $\omega(n)\rightarrow\infty$ and, for all $x\in {\mathbb R}, $
\begin{equation}
\label{uniformbound}
  \biggl|\ \mathbb{P}_n^{\{0,k\}} \left( \log \T \leq \mu_n + x \sigma_n \right) - \phi(x)\ \biggr| \leq \frac{1}{\omega(n)}.\end{equation}
However, we are not free to choose $\omega(n)$; it   is an unspecified function that could grow arbitrarily slowly.  If we impose the restriction that $|x|\leq b_{n}$,
where $b_{n}=\log^{1/2}(\sqrt{n/\lambda})$, or any other specific function we choose that is $o\left(\log^{3/4}(\sqrt{n/\lambda})\right)$, then 
Equation (\ref{approximation}) gives a better error term than the uniform bound in (\ref{uniformbound}).

\begin{proof}
It follows from Equation (\ref{splitting_probability}) that 
$$|\mathbb{P}_n^{\{0,k\}} \left( \log \T \leq \mu_n + x \sigma_n \right) - \phi(x)| =
|\zeta_1 + \zeta_2 + \zeta_3-  \phi(x)|.$$
Hence, by the triangle inequality,
\begin{equation} \label{final_estimate_for_lognormality}
|\mathbb{P}_n^{\{0,k\}} \left( \log \T \leq \mu_n + x \sigma_n \right) - \phi(x)| \leq 
|\zeta_2 - \phi(x)| +\zeta_1 + \zeta_3.
\end{equation}
Using Lemma \ref{lemma_concentration} we are able to write 
$$\phi(x) = \sum_{m \in I_2} \mathbb{P}_n^{\{ 0,k\}} (\Z = m) \phi(x) + O\left(  \exp \left( - c \log^{1/4} \left( \frac{n}{\lambda }\right)\right) \right).$$
Therefore, by the triangle inequality and Equation (\ref{zetaj}),
\begin{align*}
|\zeta_2-\phi(x)| 
\leq &
\sum_{m \in I_2}{\mathbb P}_n^{\{0,k\}} (\Z = m)
\left| \mathbb{P}_n^{\lbrace 0,k\rbrace} \left( \frac{ \log \T - \mu_n }{\sigma_n} \leq x 
\big| {\bf Z}=m \right) - \phi(x) \right| \\
& + 
 O(  \exp \left( - c \log^{1/4} \left( \frac{n}{\lambda }\right) \right).
\end{align*}
Using the definition of   $ \Delta_x(n)$   in  Lemma \ref{lemma_Pn_conditional},
we have
\begin{equation} \label{zeta2-phi}
|\zeta_2-\phi(x)| \leq  \Delta_x(n) \sum_{m \in I_2} \mathbb{P}_n^{\{0,k\}} (\Z = m) + O(  \exp \left( - c \log^{1/4} \left( \frac{n}{\lambda }\right) \right).
\end{equation}
Using Lemma \ref{lemma_Pn_conditional} in (\ref{zeta2-phi}), and putting (\ref{zeta13small}) into (\ref{final_estimate_for_lognormality}), we get
\begin{eqnarray*}
  & & |\mathbb{P}_n^{\{0,k\}} \left( \log \T \leq \mu_n 
      + x \sigma_n \right) - \phi(x)| \\
& \leq & \frac{K_2}{\log^{1/4}(\sqrt{n/\lambda})}+
       \frac{K_3|x|}{\log^{3/4}(\sqrt{n/\lambda})}+ 
       O(\exp \left( - c \log^{1/4} \left( 
       \frac{n}{\lambda }\right) \right).
\end{eqnarray*}
The last of the three terms is negligible.
\hfill
\end{proof}

The next theorem implies that, for most $\{0,k\}$-mappings on 
$n$ nodes, $\log \BB$ and $\log \T$ are approximately equal.
(Later this fact will be used to prove that $\log \BB$ is also asymptotically normal.)

\begin{theorem} \label{theorem_convergence_probability_difference}
Let $k = k(r)$ and $n = n(r)$ be sequences such that $n = kr$ and, for some $0 < \alpha < 1$, $k = o(n^{1 - \alpha})$ as $r$ approaches infinity. 
For $r \geq 1$, let $\chi_n$ be the random variable defined over $\{0,k\}$-mappings on $n$ nodes as $\chi_n = (\log \BB - \log \T)/\sigma_n$, where $\sigma_n = \frac{1}{\sqrt{3}}\log^{3/2} (\sqrt{n/\lambda})$. Then, for any $\varepsilon >0$, 
$$\mathbb{P}_n^{\{0,k\}}\left( \chi_n > \varepsilon \right) = O\left(\frac{(\log\log n)^{2}\log n}{\varepsilon \sigma_{n}}\right),$$
as $r$ approaches infinity.
\end{theorem}
\begin{proof}
Let $ {\bf D}=(\log \BB - \log \T)= \sigma_{n}\chi_{n}.$
By the Law of Total Probability we have, for any $\xi$,
\begin{eqnarray}
     {\mathbb P}_n^{\{0,k\}} ({\bf D}  > \xi )
 &=& \sum_{m = 1}^{r}{\mathbb P}_n^{\{0,k\}} (\Z = m) 
     \mathbb{P}_n^{\{0,k\}} ( {\bf D} >\xi | \Z=m) \nonumber \\
 &=& \sum_{m = 1}^{r}{\mathbb P}_n^{\{0,k\}} (\Z = m) 
     {\bf Q}_{m} ( {\bf D} >\xi). \label{plughere}
\end{eqnarray}
Define $L(1)=L(2)=1$, and $L(m)=(\log\log m)^{2}\log m$ for all $m>2$.
It is known (\cite{ArratiaTavare1992LimitDiscreteProcess}, page 333) that there is a positive constant $\kappa$ such that,
for  uniformly random permutations of $[m],$  the expected value of $\log \BB - \log \T$ is  bounded 
above by $\kappa L(m)$. 
We note that $\kappa L(m)$ is a non-decreasing function of $m$.
 It follows from Markov's Inequality (Section 3.4 of \cite{RohatgiSaleh2011ProbabilityStatistics}) that, for any $\xi >0$, and for $1\leq m\leq r$,
\[ {\bf Q}_{m}( {\bf D}  > \xi  )\leq \frac{\kappa L(m)}{\xi }   \leq  \frac{\kappa L(n)}{\xi}.\]
Putting this back into (\ref{plughere}), we get
${\mathbb P}_n^{\{0,k\}} ({\bf D}  > \xi )\leq  \frac{\kappa L(n) }{\xi}.$
In particular, with $\xi=\varepsilon \sigma_{n}$, we have
${\mathbb P}_n^{\{0,k\}} (\chi_{n}   > \varepsilon)\leq   \frac{\kappa L(n) }{\varepsilon \sigma_{n}}$.
\end{proof}

Next  we prove, using  Theorem \ref{theorem_convergence_probability_difference} and a variant of Slutsky's theorem, that
$\BB$ is asymptotically lognormal.

\begin{theorem} \label{theorem_lognormalityB}
Let $k = k(r)$ and $n = n(r)$ be sequences such that $n = kr$ and, for some $0 < \alpha < 1$, $k = o(n^{1 - \alpha})$ as $r$ approaches infinity. 
Let $\mu_n=\frac{1}{2}\log^{2}(\sqrt{n/\lambda})$, $\sigma_n^{2} = \frac{1}{3}\log^{3}(\sqrt{n/\lambda})$. Let $\BB(f)$ denote the product the cycle lengths
 of a mapping $f$ and, for $r \geq 1$, let $Y_n$ be the normalized random variable 
 $Y_n = (\log \BB - \mu_n) / \sigma_n$.  Then, the sequence defined by $Y_n$ converges in distribution to a standard normal distribution. Furthermore
\begin{eqnarray*}
 && \mathbb{P}_n^{\{0,k\}} \left( \log \BB \leq \mu_n 
     + x \sigma_n \right) \\
&=& \phi(x) + O\left(\frac{|x|}{\log^{3/4}(\sqrt{n/\lambda})}\right)
    +O\left(\frac{(\log\log n)^{2}}{\log^{1/4}(n/\lambda)}\right).
\end{eqnarray*}
\end{theorem}
\begin{proof}
One direction is trivial: $\log \T\leq \log \BB,$ so by Theorem \ref{theorem_lognormalityT},
\begin{equation}
\label{upperT7}
 \mathbb{P}_n^{\{0,k\}} \left( Y_{n} \leq x  \right)\leq \mathbb{P}_n^{\{0,k\}} \left(  X_{n} \leq  x \right)= \phi(x) +  O\left(\frac{|x|+\log^{1/2}(\sqrt{n/\lambda})}{  \log^{3/4}(\sqrt{n/\lambda})}\right).
 \end{equation}

In the other direction, we have for any $\varepsilon >0$,
\begin{eqnarray*} 
 && \mathbb{P}_n^{\{0,k\}}\left(Y_{n} \leq x \right) \\
&\geq & \mathbb{P}_n^{\{0,k\}}\left(Y_{n} \leq x 
    \text{\ and\ } \chi_{n} \leq \varepsilon\right) \\
&=& \mathbb{P}_n^{\{0,k\}} \left( X_{n}+\chi_{n}\leq x  
    \text{\ and\ } \chi_{n} \leq \varepsilon\ \right)\\
&\geq & \mathbb{P}_n^{\{0,k\}} \left( X_{n}+\varepsilon \leq x 
    \text{\ and\ } \chi_{n} \leq \varepsilon\ \right)\\
&\geq & \mathbb{P}_n^{\{0,k\}} 
    \left( X_{n} \leq x- \varepsilon \right)-
    \mathbb{P}_n^{\{0,k\}}\left(\chi_{n} > \varepsilon\right)\\
&=& \phi(x-\varepsilon)+ O\left(\frac{|x-\varepsilon|+\log^{1/2}
    (\sqrt{n/\lambda})}{  \log^{3/4}(\sqrt{n/\lambda})}\right) 
    + O\left(\frac{(\log\log n)^{2}\log n}
    {\varepsilon \sigma_{n}} \right),
\end{eqnarray*}
where in the last step we used Theorem \ref{theorem_lognormalityT} and Theorem \ref{theorem_convergence_probability_difference}. Finally, choose 
$\varepsilon=\frac{1}{\log^{1/4}(\sqrt{n/\lambda})}$, and use 
the mean value theorem to write $\phi(x-\varepsilon)=\phi(x)
+O(\varepsilon)$. Then,
\begin{eqnarray*}
 && {\mathbb P}_n^{\{0,k\}} \left( Y_{n} \leq x  \right) \\
&\geq & \phi(x)+O(\varepsilon)+  
     O\left(\frac{|x|+\log^{1/2}(\sqrt{n/\lambda})}
     {\log^{3/4}(\sqrt{n/\lambda})}\right) 
     +O\left(\frac{(\log\log n)^{2}\log n}
     {\varepsilon \sigma_{n}}\right)\\
&=& \phi(x)+ O\left(\frac{|x|}{  \log^{3/4}(\sqrt{n/\lambda})}\right) 
    +O\left(\frac{(\log\log n)^{2}}{\log^{1/4}(n/\lambda)}\right).
\end{eqnarray*}
\end{proof}

\section{Heuristics} \label{section_heuristics}

In the analysis of his rho factorization method \cite{pollard1975factorization}, Pollard conjectured that quadratic polynomials modulo large primes behave like random mappings with respect to their average rho length. 
However, it should be noted that the indegree distribution of a class of mappings impacts the asymptotic distribution of a number of parameters \cite{ArneyBender1982mappings}; the indegree distribution of a mapping $f$ on $n$ nodes is defined as the sequence $n_j = \# \{ y \in [n] \colon |f^{-1}(y)| = j \}$, $j \geq 0$. 
Since a quadratic polynomial modulo an odd prime $p$ has a very particular indegree distribution, namely $(n_0, n_1, n_2) = (\frac{p-1}{2}, 1, \frac{p-1}{2})$, one might wonder if $\{0,2\}$-mappings do not represent a better heuristic model.
Furthermore, there are classes of polynomials from which one might not expect the typical random mapping behavior, and it is possible to use different classes of mappings as heuristic models. 
This is the case for the polynomials of the form $f(x) = x^d + a \in \mathbb{F}_q[x]$, where, as usual, $\mathbb{F}_q$ denotes the finite field on $q$ elements.
Their indegree distribution satisfies
\begin{equation} \label{0k_polynomials}
n_0 = \left( 1 - \frac{1}{k} \right) (q-1), \quad n_1 = 1, \quad n_k = \frac{1}{k}(q-1),
\end{equation}
where $k = \gcd(q-1, d)$.
We note that the indegree distribution of these polynomials satisfies $n_0/q \sim 1 - 1/k$, $n_1/q = o(1)$ and $n_k/q \sim 1/k$ as $q$ approaches infinity.
We refer to the polynomials with indegree distribution (\ref{0k_polynomials}) as \emph{$\{0,k\}$-polynomials}.
As a particular case, we note that a polynomial of the form $x^k + a \in \mathbb{F}_p[x]$, $p \equiv 1 \pmod k$, is a $\{0,k\}$-polynomial.

The interest in the heuristic approximation mentioned above can be attributed at least in part to the wealth of asymptotic results on the statistics of mappings with indegree restrictions, when compared to the literature on the number theoretical setting; see for example \cite{ArneyBender1982mappings,DrmotaSoria1997RandomMappings}.
The main term of several asymptotic results on the statistics of a class $\cal F$ of mappings with restrictions on the indegrees depends on its asymptotic average coalescence $\lambda = \lambda({\cal F})$, defined as in Section \ref{section_preliminary_results}. This is the case for the rho length of a random node, parameter involved in the analysis of Pollard factorization algorithm.
Since $\lambda = 1$ for unrestricted mappings and $\{0,2\}$-mappings, these two classes represent equally accurate models for the average rho length of quadratic polynomials \cite{MartinsPanario2015heuristic}. 
It is curious that the knowledge of the indegree distribution of these polynomials does not represent an improvement on the heuristic in this case. 
It is worth noting that our asymptotic results on different classes of $\{0,k\}$-mappings are determined by their coalescence $\lambda$ as well; compare Theorems \ref{theorem_T} and \ref{theorem_BB_0,k} with Equations (\ref{schmutz_estimate_T}) and (\ref{schmutz_estimate_B}). Compare $\mu_n$ and $\mu_m^*$ with $\sigma_n$ and $\sigma_m^*$ as well, under the light of the fact that the expected number of cyclic nodes over all unrestricted or $\{0,k\}$-mappings are asymptotically equivalent to $\sqrt{\pi n/2}$ and $\sqrt{\pi n/2 \lambda}$, respectively. We note that if $\log k = o(\log n)$ then $\mu_n \sim \mu_n^*$ and $\sigma_n \sim \sigma_n^*$ as $r$ approaches infinity.
%

In this section we consider classes of $\{0,k\}$-mappings, treated in the previous sections, as heuristic models for $\{0,k\}$-polynomials. 
Our focus lies on polynomials of a certain degree modulo large prime numbers, hence from this point on we restrict our attention to $\{0,k\}$-mappings with $k \geq 2$ fixed, even though the results of the previous sections hold in a more general setting. The asymptotic results in this section are taken as $n$ approaches infinity.

\subsection{Sampling $\{0,k\}$-Mappings} \label{subsection_sampling_mappings}

In our experiments, for each prime number $p \equiv 1 \pmod k$ considered, we select $p$ $\{0,k\}$-mappings on $n = p-1$ nodes uniformly at random according to the following algorithm. For $f$ a $\{0,k\}$-mapping, let ${\cal N}_k = \{ y \in [n] \colon | f^{-1}(y) | = k \}$. We note that $|{\cal N}_k| = r$.
We determine the set ${\cal N}_k$ by selecting a permutation $\sigma = \sigma_1 \cdots \sigma_n \in S_n$ uniformly at random and defining ${\cal N}_k = \{ \sigma_1, \dots, \sigma_r \}$. The image $f(x)$ of every element $x \in [n]$ is defined by choosing again a permutation $\tau = \tau_1 \cdots \tau_n \in S_n$ uniformly at random. The first $k$ elements define the preimage of $\sigma_1$: $f^{-1}(\sigma_1) = \{ \tau_1, \dots, \tau_k \}$. The next $k$ elements determine the preimage of the element of $\sigma_2$, and so on. We make this process precise in the algorithm below.

\begin{algorithm}[H]
	\DontPrintSemicolon
   \SetAlgoLined
   \KwIn{Integers $r \geq 1$ and $k \geq 2$.} 
   \KwOut{$\{0,k\}$-mapping $f$ on $n = kr$ nodes.}
   Pick a permutation $\sigma = \sigma_1 \cdots \sigma_n \in S_n$ uniformly at random. 
   
   Pick a permutation $\tau = \tau_1 \cdots \tau_n \in S_n$ uniformly at random. \\
   \For{$i = 0, \dots, r-1$}{
   \For{$j = 1, \dots, k$}{$f( \tau[ik + j] ) = \sigma[i+1]$ 
   \tcp*[r]{\small $\tau[\ell]$ denotes $\tau_\ell$, same for $\sigma[\ell]$.}
   } 
   }
   \Return{$f$.}
   \caption{\textsc{Generating a random uniform $\{0,k\}$-mapping.}}
\end{algorithm}

\begin{theorem}
Assume that the permutations $\sigma, \tau$ in Steps 1 and 2 of Algorithm 1 are uniform random permutation of $S_n$. Then Algorithm 1 returns a uniform random $\{0,k\}$-mapping.
\end{theorem}
\begin{proof}
Let $f$ be a $\{0,k\}$-mapping on $n$ nodes and ${\cal N}_k$ as above. We note that the probability that Step 1 returns a given permutation $\sigma \in S_n$ is $1/n!$. Also, the number of permutation that define the same set ${\cal N}_k = \{ \sigma_1, \dots, \sigma_r \}$ is given by the number of permutations of 
$\{ \sigma_1, \dots, \sigma_r \}$ times the number of permutations of $\{ \sigma_{r+1}, \dots, \sigma_n \}$. Hence the probability $p_{ {\cal N}_k }$ that the set ${\cal N}_k$ is chosen in Step 1 is
\begin{equation} \label{Nk}
p_{ {\cal N}_k } = \frac{ r! (n-r)! }{n!}.
\end{equation}
Again, the probability that a given permutation is chosen in Step 2 is $1/n!$. Moreover, the number of permutations that define the same \emph{sequence of sets} ${\cal A}_i = f^{-1}(\sigma_i)$, $i = 1, \dots, r$, equals the product of the number of permutations on each $A_i$. Therefore, the probability $p_{ (A_1, \dots, A_r) }$ that $(A_1, \dots, A_r)$ is chosen satisfies
\begin{equation} \label{A1_Ar}
p_{ (A_1, \dots, A_r) } = \frac{ k! \cdots k!  }{n!} = \frac{ (k!)^r }{n!}.
\end{equation}
It follows from Equations (\ref{Nk}) and (\ref{A1_Ar}) that the probability that $f$ is returned by Algorithm 1 is $p_{ {\cal N}_k }\cdot p_{ (A_1, \dots, A_r) }= r! (n-r)! (k!)^r /(n!)^2$ which does not depend on $f$.
\hfill
\end{proof}

We discuss next the problems that can occur in the numerical estimate of the expectation of a random variable by sampling. 
To simplify notation, let $\vec{S}= f_1,f_2,f_3, \dots$ denote a sequence of independent random samples chosen uniformly at random from the class of $\{0,k\}$-mappings on $n$ nodes. 
We consider the sequence of numbers defined by $\xi = \xi(n) = \left( {\mathbb E}^{ \{0,k\} }_n(\T) \right)^a$, where $a$ depends on $n$ as well.
We define
$$\NN = \NN(n,\vec{S},a)=\min\lbrace  t: \T(f_{t}) \geq \xi \rbrace.$$
Thus $\NN$  has a geometric distribution: for $j \geq 1$ we have 
\begin{equation} \label{distribution_N}
\begin{array}{rcl}
\mathbb{P}_n^{ \{0,k\} }(\NN = j) & = & \mathbb{P}^{ \{0,k\} }_n(\T \geq \xi) \cdot \mathbb{P}^{ \{0,k\} }_n(\T < \xi)^{j-1}, \\[5pt]
{\mathbb E}_n^{ \{0,k\} }(\NN) & = & \displaystyle \frac{1}{\mathbb{P}_n(\T \geq \xi)}.
\end{array}
\end{equation}
We note that if $a = \log^{-1/4} n$, then the ratio between $\xi$ and ${\mathbb E}^{ \{0,k\} }_n(\T)$ approaches zero as $n \to \infty$. We prove in Theorem \ref{theorem_first_hit_T} that this particular choice of $a$ defines a random variable $\NN$ whose expectation has exponential growth. We remember that in this section we assume that $k\geq 2$ is a fixed integer.

\begin{lemma} \label{lemma_applyLandau}
Let $f$ be a $\{0,k\}$-mapping on $n = kr$ nodes and let $\alpha=\frac{n^{2/3}}{\log^{3} n}$. If $a^{-1} \log^{-1/3} n = o(1)$ then, for sufficiently large $n$, $\Z(f) \leq \alpha$ implies $\T(f) < \xi$.
\end{lemma}  
\begin{proof}
If $f: [n] \to [n]$ is a $\{0,k\}$-mapping and $\Z (f) =m$, then clearly  $\T(f) \leq \max\limits_{\sigma \in S_{m} }\T (\sigma)$.
It follows from Landau's theorem \cite{landau1909handbuch, szalay1980MaximalOrder} that 
$$\max\limits_{\sigma \in S_{m} }\T (\sigma)= \exp \left( \sqrt{m\log m}(1+o(1)) \right),$$
hence, for sufficiently large $n$,
\begin{equation} \label{max_T_first_hit}
\T(f) < \exp \left( 2 \sqrt{m\log m} \right). 
\end{equation}
We note that, if $m\leq \alpha$, then 
$$m \log m \leq \alpha \log \alpha = \frac{ n^{2/3} }{ \log^3 n } \left( \frac{2}{3} \log n - \log \log^3 n \right)
\leq \frac{2}{3} \frac{ n^{2/3} }{ \log^2 n }.$$
Therefore,
\begin{equation} \label{final_estimate_first_hit}
2 \sqrt{m\log m} \leq 2 \frac{n^{1/3}}{\log n},
\end{equation}
where Theorem \ref{theorem_T} and the definition of $\xi$ and $\alpha$ imply
$$\frac{ 2 n^{1/3} \log^{-1} n }{\log \xi} = O \left( a^{-1} \log^{-1/3} n \right).$$
If $a^{-1} \log^{-1/3} n = o(1)$, then $\exp \left( 2 n^{1/3} \log^{-1} n \right) < \xi$ for sufficiently large $n$.
The result follows from Equations (\ref{max_T_first_hit}) and (\ref{final_estimate_first_hit}).
\hfill
\end{proof}

\begin{theorem} \label{theorem_first_hit_T}
If $a^{-1} \log^{-1/3} n = o(1)$ then, for sufficiently large $n$, we have  
$${\mathbb E}_n^{ \{0,k\} }(\NN) > \exp \left( \frac{ \lambda  n^{1/3}}{3\log ^{6}n } \right),$$
and, in addition,
$${\mathbb P}_n^{ \{0,k\} } \left( \NN \leq \exp \left( \frac{ \lambda n^{1/3}}{4\log ^{6}n } \right) \right) \leq \exp \left( -\frac{ \lambda  n^{1/3}}{12\log ^{6}n } \right).$$
\end{theorem} 
\begin{proof}
Let $\alpha=\frac{n^{2/3}}{\log^{3} n}$. 
Using first the Law of Total Probability we obtain
\begin{eqnarray*}
    {\mathbb P}_{n}^{ \{0,k\} }(\T \geq \xi) 
&=& {\mathbb P}_{n}^{ \{0,k\} }(\T \geq \xi | \Z\leq \alpha ) 
    \cdot {\mathbb P}_{n}^{ \{0,k\} }(\Z\leq \alpha ) + \\
 && {\mathbb P}_{n}^{ \{0,k\} }(\T \geq \xi | \Z> \alpha ) 
    \cdot {\mathbb P}_{n}^{ \{0,k\} }(\Z > \alpha ).
\end{eqnarray*}
We note that Lemma \ref{lemma_applyLandau} implies ${\mathbb P}_{n}^{ \{0,k\} }(\T \geq \xi | \Z \leq \alpha ) = 0$ for $r$ large enough. Since ${\mathbb P}_{n}^{ \{0,k\} }(\T \geq \xi | \Z> \alpha ) \leq 1$, it follows that 
\begin{equation} \label{bound_p_tail}
\mathbb{P}_n^{ \{0,k\} }(\T \geq \xi) \leq {\mathbb P}_{n}^{ \{0,k\} }(\Z > \alpha ).
\end{equation}
We have by Lemma \ref{lemma_unimodal} that $\alpha$ is greater than the mode $m_\#$ of $\Z$, thus using Equation (\ref{bound_p_tail}) one obtains 
$$\mathbb{P}_n^{ \{0,k\} }(\T \geq \xi) \leq \sum_{m = \lceil \alpha \rceil}^n {\mathbb P}_{n}^{ \{0,k\} }(\Z = m) \leq 
\sum_{m = \lfloor \alpha \rfloor}^n {\mathbb P}_{n}^{ \{0,k\} }(\Z = \lfloor \alpha \rfloor),$$
hence,
\begin{equation} \label{bound_p}
\mathbb{P}_n^{ \{0,k\} }(\T \geq \xi) \leq n {\mathbb P}_{n}^{ \{0,k\} }(\Z=\lfloor \alpha \rfloor ).
\end{equation}
We note that Lemma \ref{lemma_asymptotic_distribution_Z_0,k} implies
\begin{align*}
n {\mathbb P}_{n}^{ \{0,k\} }(\Z = \lfloor \alpha \rfloor) 
& = 
\lambda \frac{ n^{2/3} }{ \log^3 n } \exp \left( - \frac{\lambda}{2n} \frac{ n^{4/3} }{ \log^6 n } + 
O \left( \frac{1}{n^2} \frac{ n^2 }{ \log^9 n } \right) + o(1) \right) \\
& =
\exp \left( - \frac{ \lambda n^{1/3} }{ 2 \log^6 n } + \log \lambda + \frac{2}{3} \log n - \log \log^3 n + o(1) \right),
\end{align*}
hence, for sufficiently large $n$,
\begin{equation} \label{bound_n_times_tail_of_Z}
n {\mathbb P}_{n}^{ \{0,k\} }(\Z = \lfloor \alpha \rfloor) < 
\exp \left( - \frac{ \lambda n^{1/3} }{ 3 \log^6 n } \right).
\end{equation}
The bound for $\mathbb{E}_n^{ \{0,k\} }(\NN)$ follows at once from Equations (\ref{distribution_N}), (\ref{bound_p}) and (\ref{bound_n_times_tail_of_Z}).

We note that, for any positive integer $s$, Equation (\ref{distribution_N}) implies
$${\mathbb P}_{n}^{ \{0,k\} }(\NN \leq s) = \sum_{1 \leq j \leq s} {\mathbb P}_{n}^{ \{0,k\} }(\NN = j) \leq 
\sum_{1 \leq j \leq s} \mathbb{P}_n^{ \{0,k\} }(\T \geq \xi) \mathbb{P}_n^{ \{0,k\} }(\T < \xi)^{j-1},$$
so ${\mathbb P}_{n}^{ \{0,k\} }(\NN \leq s) \leq s \mathbb{P}_n^{ \{0,k\} }(\T \geq \xi)$.
In particular, if $s = \exp \left( \frac{ \lambda n^{1/3}}{4\log ^{6}n } \right)$, then Equations (\ref{bound_p}) and (\ref{bound_n_times_tail_of_Z}) imply
$${\mathbb P}_{n}^{ \{0,k\} } \left( \NN \leq \exp \left( \frac{ \lambda n^{1/3} }{ 4\log ^{6}n } \right) \right) \leq  
\exp \left( \left( \frac{1}{4} - \frac{1}{3} \right) \frac{ \lambda n^{1/3} }{\log^6 n} \right) =
\exp \left( - \frac{ \lambda n^{1/3} }{12 \log^6 n} \right),$$
as desired.
\hfill
\end{proof}

Let $\widetilde{\xi} = \left( \mathbb{E}_n^{ \{0,k\} }(\BB) \right)^b$ where $b = b(n)$.
We prove, as a particular case of Theorem \ref{theorem_first_hit_B} below, that if $b = \log^{-1} n$ then the expected number of random mappings $f$ sampled before encountering one such that $\BB(f) \geq \widetilde{\xi}$ is exponentially large.
As before, we consider a sequence $\vec{S}= f_1,f_2,f_3, \dots$ of $\{0,k\}$-mappings on $n$ nodes chosen independently and uniformly at random. We estimate the asymptotic behavior of the random variable $\widetilde{\NN} = \min\lbrace t: \BB (f_{t}) \geq \widetilde{\xi} \rbrace$.

\begin{lemma} \label{lemma_CboundCor}
Let $f$ be a $\{0,k\}$-mapping on $n = kr$ nodes and let $\widetilde{\alpha} = \frac{ n^{1/3} }{ \log^{3} n }$. If $b^{-1} \log^{-2} n = o(1)$ then, for sufficiently large $n$, $\CC(f) < \widetilde{\alpha}$ implies $\BB(f) < \widetilde{\xi}$.
\end{lemma}
\begin{proof}
If $\Z(f) = m$, then $\BB$ is the product of $\CC$ positive numbers whose sum is $m$.
Hence $\BB \leq (m/\CC)^{\CC} \leq (n/\CC)^{\CC}.$ The function $C \mapsto (n/C)^{C}$ is increasing for $C\leq \widetilde{\alpha}$. Therefore, when $\CC < \widetilde{\alpha}$, we have for all sufficiently large $n$ that
\begin{equation} \label{final_estimate_first_hit_B}
\BB < \left( n/ \widetilde{\alpha} \right)^{\widetilde{\alpha}} = \exp \left( \widetilde{\alpha} \log (n/ \widetilde{\alpha}) \right).
\end{equation}
It follows from Theorem \ref{theorem_BB_0,k} and the definition of $\widetilde{\xi}$ and $\widetilde{\alpha}$ that
$$\frac{ \widetilde{\alpha} \log (n/ \widetilde{\alpha}) }{ \log \widetilde{\xi} } = 
O \left( \frac{ n^{1/3} }{ \log^{3} n } \log \left( \frac{ n^{1/3} }{ \log^{3} n } \right) b^{-1} n^{-1/3} \right) = 
O \left( b^{-1} \log^{-2} n \right).$$
If $b^{-1} \log^{-2} n = o(1)$ then, for sufficiently large $n$, we have 
$\exp \left( \widetilde{\alpha} \log (n/\widetilde{\alpha}) \right) < \widetilde{\xi}$. The result follows from Equation (\ref{final_estimate_first_hit_B}).
\hfill
\end{proof}

\begin{theorem} \label{theorem_first_hit_B}
If $b^{-1} \log^{-2} n = o(1)$, then there exist positive constants $c_1, c_2$ such that, for sufficiently large $n$,
$$\mathbb{E}_n^{ \{0,k\} }(\widetilde{\NN}) > \exp \left( c_1 \frac{ (n/\lambda)^{1/3} }{ \log^{3} (n/\lambda) } \right),$$
and, in addition, 
$${\mathbb P}_{n}^{ \{0,k\} } 
\left( \widetilde{\NN} \leq \exp \left( c_2 \frac{ (n/\lambda)^{1/3} }{ \log^{3} (n/\lambda) } \right) \right) \leq 
\exp \left( - c_2 \frac{ (n/\lambda)^{1/3} }{ \log^{3} (n/\lambda) } \right),$$
\end{theorem}
\begin{proof}
The random variable $\widetilde{\NN}$ has a geometric distribution, hence
\begin{equation} \label{expectation_N_prime}
\mathbb{E}_n^{ \{0,k\} }( \widetilde{\NN} ) =\frac{1}{ \mathbb{P}_n^{ \{0,k\} }( \BB \geq \widetilde{\xi}) },
\end{equation}
where, by the Law of Total Probability, 
\begin{eqnarray*}
    \mathbb{P}_n^{ \{0,k\} }( \BB \geq \widetilde{\xi}) 
&=& \mathbb{P}_n^{ \{0,k\} }( \BB \geq \widetilde{\xi} | {\bf C} 
    \geq \widetilde{\alpha} ) \mathbb{P}_n^{ \{0,k\} }({\bf C} 
    \geq  \widetilde{\alpha}) + \\
 && \mathbb{P}_n^{ \{0,k\} }( \BB  \geq  \widetilde{\xi} | {\bf C} 
    < \widetilde{\alpha})\mathbb{P}_n^{ \{0,k\} }
    ( {\bf C} <\widetilde{\alpha}).
\end{eqnarray*}
As a consequence of Lemma \ref{lemma_CboundCor} we have $\mathbb{P}_n^{ \{0,k\} }( \BB \geq \widetilde{\xi} | {\bf C} < \widetilde{\alpha} ) = 0$, so $\mathbb{P}_n^{ \{0,k\} }( \BB \geq \widetilde{\xi} | {\bf C} \geq \widetilde{\alpha} ) \leq 1$ implies
\begin{equation} \label{probability_B_greater_xi}
\mathbb{P}_n^{ \{0,k\} }( \BB \geq \widetilde{\xi}) \leq \mathbb{P}_n^{ \{0,k\} }(  {\bf C} \geq  \widetilde{\alpha}).
\end{equation}
Using again the Law of Total Probability one obtains
\begin{equation} \label{probability_C_greater_alpha}
\begin{split}
\mathbb{P}_n^{ \{0,k\} }(  {\bf C} \geq  \widetilde{\alpha} )
& = 
\sum\limits_{m}\mathbb{P}_n^{ \{0,k\} }(  {\bf C} \geq  \widetilde{\alpha} | \Z =m) \mathbb{P}_n^{ \{0,k\} }( \Z=m) \\
& = 
\sum\limits_{m}{\mathbf Q}_{m}({\bf C} \geq  \widetilde{\alpha}) \mathbb{P}_n^{ \{0,k\} }( \Z=m).
\end{split}
\end{equation}
Using moment generating functions \cite{FlajoletSoria1993GeneralSchemas} one is able to prove that there exist constants 
$\rho > 1$ and $c_0 > 0$ such that ${\bf Q}_{m}(\CC \geq \widetilde{\alpha}) \leq c_0 \rho^{-\widetilde{\alpha}}$ for all $1 \leq m \leq n$. Thus 
Equations (\ref{probability_B_greater_xi}) and (\ref{probability_C_greater_alpha}) imply
$$\mathbb{P}_n^{ \{0,k\} }( \BB \geq \widetilde{\xi}) \leq c_0 \rho^{-\widetilde{\alpha}} \sum_{m = 1}^n \mathbb{P}_n^{ \{0,k\} }( \Z=m) \leq c_0 \rho^{-\widetilde{\alpha}}.$$
It follows from Equation (\ref{expectation_N_prime}) that 
\begin{equation} \label{first_hit_B_proof}
\mathbb{E}_n^{ \{0,k\} }(\widetilde{\NN}) \geq \exp \left( - \log c_0 + \frac{ (n/\lambda)^{1/3} }{ \log^{3} (n/\lambda) } \log \rho \right) \geq \exp \left( c_1 \frac{ (n/\lambda)^{1/3} }{ \log^{3} (n/\lambda) } \right),
\end{equation}
for $c_1 = (\log \rho)/2$ and sufficiently large $n$.

We conclude the proof with an argument analogous to the one in the proof of Theorem \ref{theorem_first_hit_T}.
For any positive integer $s$ we have
$${\mathbb P}_{n}^{ \{0,k\} }(\widetilde{\NN} \leq s) = \sum_{1 \leq j \leq s} {\mathbb P}_{n}^{ \{0,k\} }(\widetilde{\NN} = j) \leq 
\sum_{1 \leq j \leq s} \mathbb{P}_n^{ \{0,k\} }( \BB \geq \widetilde{\xi}) \mathbb{P}_n^{ \{0,k\} }( \BB < \widetilde{\xi})^{j-1},$$
so ${\mathbb P}_{n}^{ \{0,k\} }(\widetilde{\NN} \leq s) \leq s \mathbb{P}_n^{ \{0,k\} }( \BB \geq \widetilde{\xi})$.
It follows from Equations (\ref{expectation_N_prime}) and (\ref{first_hit_B_proof}) that, for $s = \exp \left( c_2 \frac{ (n/\lambda)^{1/3} }{ \log^{3} (n/\lambda) } \right)$ and $c_2 = c_1/2$,
$${\mathbb P}_{n}^{ \{0,k\} } 
\left( \widetilde{\NN} \leq \exp \left( c_2 \frac{ (n/\lambda)^{1/3} }{ \log^{3} (n/\lambda) } \right) \right) \leq  
\exp \left( (c_2 - c_1) \frac{ (n/\lambda)^{1/3} }{ \log^{3} (n/\lambda) } \right),$$
where $c_2 - c_1 = - c_2$.
\hfill
\end{proof}

\subsection{Numerical Results} \label{subsection_numerical_results}

We exhibit in Table \ref{table_numerical_results} our numerical results on the behavior of $\T$ and $\BB$ over different classes of polynomials over finite fields and different classes of mappings. 
For each value of $k$, we consider the first $100$ primes greater than $10^3$ of the form indicated in Table \ref{table_numerical_results}. For each such prime, we select, according to Algorithm 1, $p$ mappings on $n = p-1$ nodes; we also consider all $p$ polynomials of the form indicated in Table 1.
We compute the exact value of $\T$ for each function and compute the corresponding average values $\overline{\T}(p)$. We compute the ratio $R_\T(p)$ between $\log \overline{\T}(p)$ and the quantity in Theorem \ref{theorem_T}. 
In Table \ref{table_numerical_results} we exhibit the average value $\overline{R_\T}$ of $R_\T(p)$ over all primes considered; we stress the dependence of this calculation on the coalescence $\lambda$ of the corresponding class by adopting the notation $\overline{R_\T}(\lambda)$. The same is done for the parameter $\BB$.

\begin{table}
\centering
\begin{tabular}{cccc}
\hline
Class of functions 	& Asymptotic Coalescence & $\overline{R_\T}(\lambda)$ & $\overline{R_\BB}(\lambda)$ \\ \hline
Unrestricted mappings & $1$ & $0.8090$ & $0.7247$ \\
$\{0,2\}$-mappings    & $1$ & $0.7929$ & $0.7097$ \\
$x^2 + a \in \mathbb{F}_p[x]$ 						& $1$ & $0.8031$ & $2.4183$ \\
$x^4 + a \in \mathbb{F}_p[x]$, $p \equiv 3 \pmod 4$ & $1$ & $0.8033$ & $3.9237$ \\
$\{0,3\}$-mappings    & $2$ & $0.7700$ & $0.7043$ \\ 
$x^3 + a \in \mathbb{F}_p[x]$, $p \equiv 1 \pmod 3$ & $2$ & $0.7631$ & $2.5067$ \\
$\{0,4\}$-mappings    & $3$ & $0.7436$ & $0.7007$ \\
$x^4 + a \in \mathbb{F}_p[x]$, $p \equiv 1 \pmod 4$ & $3$ & $0.7391$ & $2.6055$ \\
$\{0,5\}$-mappings    & $4$ & $0.7465$ & $0.7041$ \\
$x^5 + a \in \mathbb{F}_p[x]$, $p \equiv 1 \pmod 5$ & $4$ & $0.7435$ & $3.3597$ \\
$\{0,6\}$-mappings    & $5$ & $0.6986$ & $0.6789$ \\
$x^6 + a \in \mathbb{F}_p[x]$, $p \equiv 1 \pmod 6$ & $5$ & $0.6989$ & $1.3522$ \\  \hline
\end{tabular}
\caption{Experimental results on mappings and polynomials according to their coalescence.} 
\label{table_numerical_results}
\end{table}

It is not surprising to have the ratio $\overline{R_\T}$ distant from $1$ even in the case of $\{0,k\}$-mappings, where we have an asymptotic result proved on the logarithm of the expectation of $\T$. It is proved in Theorem \ref{theorem_first_hit_T} that most of the contribution to $\mathbb{E}^{ \{0,k\} }_n (\T)$ comes from a relatively small set of exceptional maps. Unless the number of samples is enormous, as stated in the first part of the theorem, none of these exceptional maps is likely to be sampled, so our empirical estimate for $\mathbb{E}^{ \{0,k\} }_n (\T)$ is likely to be poor.
The ratios $\overline{R_\T}$ appear to decrease as $\lambda$ grows large, but this agrees, in a way, with the fact that the upper bound in Theorem \ref{theorem_first_hit_T} decreases as $k$ grows large.

Regardless of the sampling problem explained in Section \ref{subsection_sampling_mappings}, it is remarkable that the ratio between any two entries in the table above for $\overline{R_\T}$ with the same value of $\lambda$ lies in the interval $(0.97, 1.03)$. This suggests that the behavior of a typical $\{0,k\}$-polynomial can be approximated by the behavior of a typical $\{0,k\}$-mapping. However, one must be careful when using the asymptotic estimate in Theorem \ref{theorem_T} as a reference, due to the results in Theorem \ref{theorem_first_hit_T}.
The numerical results for the parameter $\BB$, on the other hand, represent a different scenario, where the ratio between numerical results for classes with the same value of asymptotic coalescence were found to be as high as $4.8835$. 
It is interesting but not clear why the heuristic performs so poorly in the approximation of the statistics of polynomials by mappings in the case of the parameter $\BB$.

\end{document}